\documentclass[10pt,psamsfonts]{article}

\usepackage{amsmath, amsfonts, amsthm, amssymb}
\usepackage{thmtools,slashed}
\usepackage{xcolor}
\usepackage{enumerate}
\usepackage{mathrsfs} 

\usepackage[colorlinks = true,
            linkcolor = blue,
            urlcolor  = cyan,
            citecolor = red,
            anchorcolor = blue,
            pagebackref]{hyperref}
            
\usepackage[alphabetic,backrefs]{amsrefs}
 
 \usepackage[parfill]{parskip}
 \usepackage{anysize}
 \marginsize{1in}{1in}{1in}{1in}

\DeclareFontFamily{U}{mathx}{}
\DeclareFontShape{U}{mathx}{m}{n}{<-> mathx10}{}
\DeclareSymbolFont{mathx}{U}{mathx}{m}{n}
\DeclareMathAccent{\widehat}{0}{mathx}{"70}
\DeclareMathAccent{\widecheck}{0}{mathx}{"71}
 
\begingroup
 \makeatletter
 \@for\theoremstyle:=definition,remark,plain\do{%
 \expandafter\g@addto@macro\csname th@\theoremstyle\endcsname{%
 \addtolength\thm@preskip\parskip
 }%
 }
 \endgroup

\declaretheorem[name=Theorem,numberwithin=section]{thm}
\declaretheorem[name=Proposition,numberlike=thm]{prop}
\declaretheorem[name=Lemma,numberlike=thm]{lemma}
\declaretheorem[name=Corollary,numberlike=thm]{cor}

\declaretheorem[name=Remark,style=definition,qed=$\blacktriangle$,numberlike=thm]{rmk}

\newcounter{noteCounter}
\newcommand{\dib}[1]{\dfrac{\partial}{\partial{#1}}}
\newcommand{\dibb}[2]{\dfrac{\partial{#1}}{\partial{#2}}}

\newcommand{\dbydr}{\tfrac{\partial}{\partial r}} 

\newcommand{\nr}{\nabla_{\frac{\partial}{\partial r}}}

\newcommand{\tdiv}{\operatorname{div}}
\newcommand{\curl}{\operatorname{curl}}
\newcommand{\rank}{\operatorname{rank}}
\newcommand{\se}{\mathsf{e}}
\newcommand{\sv}{\mathsf{v}}
\newcommand{\re}{\mathrm{e}}
\newcommand{\rj}{\mathbf{j}}

\newcommand{\Lie}{\mathcal L}
\newcommand{\hstar}{\star} 

\newcommand{\phir}{\overline{\varphi}} 
\newcommand{\psir}{\overline{\psi}} 

\newcommand{\phii}{{\varphi}} 
\newcommand{\vR}{{\mathsf R}} 
\newcommand{\rJ}{{\mathrm J}} 
\renewcommand{\Re}{\operatorname{Re}} 
\renewcommand{\Im}{\operatorname{Im}} 
\newcommand{\gN}{\widecheck{g}} 
\newcommand{\gdn}{\widecheck{g}} 
\newcommand{\nablaN}{\widecheck{\nabla}} 
\newcommand{\ri}{{\mathrm i}}

\newcommand{\R}{\mathbb R}
\newcommand{\C}{\mathbb C}

\newcommand{\PR}{\mathbb P}

\newcommand{\G}{\mathrm{G}_2}

\newcommand{\Spin}[1]{\mathrm{Spin}(#1)}
\DeclareMathOperator\vol{vol}

\newcommand{\vu}{\boldsymbol{u}}
\newcommand{\vF}{\boldsymbol{F}}

\newcommand{\scrU}{\mathscr{U}}

\newcommand{\hk}{\mathbin{\! \hbox{\vrule height0.3pt width5pt depth 0.2pt \vrule height5pt width0.4pt depth 0.2pt}}}

\makeatletter
\let\c@equation\c@thm
\makeatother
\numberwithin{equation}{section}

\begin{document}

\title{Cohomogeneity one solitons for the \\ isometric flow of $\G$-structures}

\author{Thomas A. Ivey \\ \emph{Department of Mathematics, College of Charleston} \\ \tt{iveyt@cofc.edu} \and Spiro Karigiannis \\ \emph{Department of Pure Mathematics, University of Waterloo} \\ \tt{karigiannis@uwaterloo.ca}}

\maketitle

\begin{abstract}
We consider the existence of cohomogeneity one solitons for the isometric flow of $\G$-structures on the following classes of torsion-free $\G$-manifolds: the Euclidean $\R^7$ with its standard $\G$-structure, metric cylinders over Calabi--Yau $3$-folds, metric cones over nearly K\"ahler $6$-manifolds, and the Bryant--Salamon $\G$-manifolds. In all cases we establish existence of global solutions to the isometric soliton equations, and determine the asymptotic behaviour of the torsion. In particular, existence of shrinking isometric solitons on $\R^7$ is proved, giving support to the likely existence of type I singularities for the isometric flow. In each case, the study of the soliton equation reduces to a particular nonlinear ODE with a regular singular point, for which we provide a careful analysis. Finally, to simplify the derivation of the relevant equations in each case, we first establish several useful Riemannian geometric formulas for a general class of cohomogeneity one metrics on total spaces of vector bundles which should have much wider application, as such metrics arise often as explicit examples of special holonomy metrics.
\end{abstract}

\tableofcontents

\section{Introduction} \label{introsec}

Let $(M^7, \phii)$ be a $7$-manifold equipped with a $\G$-structure $\phii$, inducing a Riemannian metric $g$, an orientation, and a Hodge dual $4$-form $\psi = \hstar_{\phii} \phii$. The torsion tensor $T$ of $\phii$ is a $2$-tensor which measures the failure of $\phii$ to be parallel, and hence the failure of the induced metric $g$ to have holonomy contained in $\G$. As there are many $\G$-structures on $M$ which induce the same metric $g$, it is natural to consider a ``best'' representative of the \emph{isometry class} $[\phii]_g$ of $\G$-structures on $M$ inducing the metric $g$. This problem was first considered by Grigorian~\cite{Grigorian-Oct}. He showed that for $M$ compact, if we \emph{restrict to the isometry class $[\phii]_g$}, then the negative gradient flow of the functional $\int_M |T|^2 \vol_g$ yields the geometric flow
\begin{equation} \label{eq:IF}
\dibb{\varphi}t = (\tdiv T) \hk \psi,
\end{equation}
called the \emph{isometric flow} of $\G$-structures. In particular, the Euler--Lagrange equation for critical points of this functional, restricted to the isometry class $[\phii]_g$, is the equation $\tdiv T = 0$, corresponding to \emph{divergence-free torsion}. It makes sense to consider the isometric flow~\eqref{eq:IF} even in the noncompact case, as a natural geometric flow of $\G$-structures all inducing the same Riemannian metric $g$.

Analytic aspects of the isometric flow have been studied by various authors. Bagaglini~\cite{Bag} established short-time existence and uniqueness. Both Grigorian~\cite{G-IF} and Dwivedi--Gianniotis--Karigiannis~\cite{DGK} derived Shi-type estimates giving bounds on higher derivatives of the torsion in terms of bounds on the torsion, proved almost monotonicity formulas and $\epsilon$-regularity theorems, and established stability (long time existence and convergence to a $\G$-structure with divergence-free torsion given initial conditions with sufficiently small torsion). Grigorian~\cite{G-IF} also obtained more refined stability results when the isometry class $[\phii]_g$ of the initial $\G$-structure admits a torsion-free representative. Dwivedi--Gianniotis--Karigiannis~\cite{DGK} also considered the structure of the singular set and established the relation between type-I singularities and shrinking solitons. (See~\cite{Grigorian-IF-survey} for a survey of results on the isometric flow.)

Recent work of Fadel--Loubeau--Moreno--S\'a Earp~\cite{FLMS} generalizes the isometric flow to any $H$-structure on a Riemannian $n$-manifold, where $H$ is a Lie subgroup of $\mathrm{SO}(n)$. Dwivedi--Loubeau--S\'a Earp~\cite{DLS} formulate a theory of isometric solitons and self-similarity for general $H$-structures, and~\cite{FLMS} extends this formulation to a larger class of flows, while Fowdar--S\'a Earp~\cite{FS} construct isometric solitons for quaternionic-K\"ahler structures on $\R^8$. Finally, in another recent preprint, Dwivedi--Gianniotis--Karigiannis~\cite{Flows2} show that a natural coupling of the isometric flow with the Ricci flow of metrics has good short-time existence properties.

The present paper is concerned with the analysis of particular examples of solitons for the isometric flow when the initial $\G$-structure $\phii_0$ is torsion-free, and such that the soliton solutions to the flow are of \emph{cohomogeneity one}. More precisely, we consider the following torsion-free $\G$-structures:
\begin{enumerate}[(i)] \setlength\itemsep{-1mm}
\item The standard $\G$-structure on the Euclidean $\R^7$.
\item A metric cylinder $(L \times N, dr^2 + g_N)$ over a Calabi--Yau $3$-fold $N$, where $L = \R$ or $S^1$.
\item A metric cone $(\R^{+} \times N, dr^2 + r^2 g_N)$ over a nearly K\"ahler $6$-manifold $N$.
\item The Bryant--Salamon $\G$-manifolds $\Lambda^2_- (M^4)$ and $\slashed{S} (S^3)$, where $M^4 = S^4$ or $\C\PR^2$.
\end{enumerate}
In each of these cases, we consider solitons for the isometric flow whose data depends only on the coordinate $r$. (In (i) the coordinate $r$ is the radial distance to the origin, while in (iv) the coordinate $r$ is the radial distance to the origin in the fibres.)

The isometric flow soliton equation is a system of PDE for a triple $(\varphi, Y, c)$ where $\varphi$ is a $\G$-structure in a given isometric class, $Y$ is a vector field, and $c \in \R$. The system of PDE is
\begin{equation} \label{eq:PDE-intro}
\Lie_Y g = 2c \, g, \qquad \tdiv T = - \tfrac{1}{2} \curl_{\varphi} Y +Y \hk T,
\end{equation}
where $T$ is the torsion and $\curl_{\varphi}$ is the curl operator of $\varphi$. The soliton is called shrinking, steady, or expanding if $c$ is positive, zero, or negative, respectively. (But see Remark~\ref{rmk:soliton-notation} for clarification.) 

In all of the above cases, we prove existence of globally defined solutions to the system~\eqref{eq:PDE-intro}, and we investigate the asymptotic behaviour of the torsion. Specifically, we prove:
\begin{enumerate}[(i)] \setlength\itemsep{-1mm}
\item On Euclidean $\R^7$ there exist global radially symmetric solutions for any $c \in \R$, all with $Y = c r \frac{\partial}{\partial r}$. For $c \leq 0$, we have $|T| \to 0$ at infinity, while for $c > 0$ we have $|T| \to \infty$ at infinity. (See Theorem~\ref{thm:R7final} and Proposition~\ref{prop:asymptotic-torsion}.)
\item On a metric cylinder $(L \times N, dr^2 + g_N)$ over a Calabi--Yau $3$-fold $N$, the only cohomogeneity one solutions have $c=0$, and $Y = b \frac{\partial}{\partial r}$ for some constant $b$. When $L = S^1$, the $\G$-structure $\varphi$ is torsion-free. When $L = \R$, the $\G$-structure $\varphi$ is \emph{not} torsion-free. The pointwise norm $|T|$ of the torsion goes to zero at one end of $\R$ and is unbounded at the other end. (See Theorem~\ref{thm:CYfinal}.)
\item On a metric cone $(\R^{+} \times N, dr^2 + r^2 g_N)$ over a nearly K\"ahler $6$-manifold $N$, the cohomogeneity one isometric solitons are exactly the same as case (i). Indeed, (i) is a special case of (iii). (See Proposition~\ref{prop:NK-solitons} and Remark~\ref{rmk:NK}.)
\item On the Bryant--Salamon $\G$-manifolds $\Lambda^2_- (M^4)$ and $\slashed{S} (S^3)$, the only cohomogeneity one solutions have $c=0$ and $Y=0$, and satisfy $|T| \to 0$ at infinity. (See Proposition~\ref{prop:BS-solitons} and Theorem~\ref{thm:BS-solitons-final}.)
\end{enumerate}
In all cases, the study of~\eqref{eq:PDE-intro} reduces to a particular nonlinear ODE with a regular singular point. These all belong to a general family, for which we present a careful analysis in Section~\ref{ODEsec}. Moreover, to simplify the derivation of the relevant equations in each case, in Section~\ref{Laplacian-sec} we establish several useful Riemannian geometric formulas for a general class of cohomogeneity one metrics on total spaces of vector bundles. These formulas should have much wider application, as such metrics arise often as explicit examples of special holonomy metrics.

\begin{rmk} \label{rmk:singularities}
When $(M, \phir)$ is $\R^7$ equipped with the standard torsion-free $\G$-structure inducing the Euclidean metric, then it was shown in~\cite[Remark 5.21]{DGK} that \emph{the non-existence of shrinking isometric solitons on $(\R^7, \phir)$ would imply that type I singularities of the isometric flow do not occur}. This is because appropriately rescaling about a type I singularity of the isometric flow yields a shrinking isometric soliton on the Euclidean $\R^7$ with the standard $\G$-structure.

As we prove in Theorem~\ref{thm:R7final} that such shrinking isometric solitons do indeed exist, the occurrence of type I singularities cannot be excluded. Indeed, in~\cite[Example 2.17]{FLMS}, the first known example of a finite-time singularity for the isometric flow is presented. It is not clear if this singularity is of type I, but if it is, then it would be interesting to see which shrinking isometric soliton it induces on $\R^7$.
\end{rmk}

\textbf{Notation.} Let $(M, g)$ be a Riemannian manifold with arbitrary local frame $\{ w_a : 1 \leq a \leq \dim M \}$ for $TM$. For any tensor $A$ on $M$, the notation $\nabla_a \nabla_b A$ means $(\nabla (\nabla A))(w_a, w_b, \cdot)$, and thus
\begin{align*}
\nabla_a \nabla_b A & = (\nabla (\nabla A))(w_a, w_b, \cdot) = \big( \nabla_{a} (\nabla A) \big) (w_b, \cdot) \\
& = w_a \big( (\nabla A)(w_b, \cdot) \big) - (\nabla A)(\nabla_a w_b, \cdot) - (\nabla A)(w_b, \nabla_a (\cdot) ) \\
& = w_a \big( (\nabla_b A)(\cdot) \big) - (\nabla_{\nabla_a w_b}) A)(\cdot) - (\nabla_b A) (\nabla_a (\cdot) ) \\
& = \big( \nabla_a (\nabla_b A) \big) (\cdot) - (\nabla_{\nabla_a w_b} A) (\cdot). 
\end{align*}
In particular, the \emph{analyst's Laplacian} $\Delta A$ is given by
\begin{equation} \label{eq:rough-Lap}
\Delta A = g^{ab} \nabla_a \nabla_b A = g^{ab} \big( \nabla_a (\nabla_b A) - \Gamma^c_{ab} \nabla_c A).
\end{equation}
This differs by a sign from the \emph{rough Laplacian} $\nabla^* \nabla A = - g^{ab} \nabla_a \nabla_b A$. Note that we work with torsion-free $\G$-structures, which are Ricci-flat, so the rough Laplacian equals the Hodge Laplacian on $1$-forms and functions.

\textbf{Acknowledgements} The research of the second author is supported by a Discovery Grant from the Natural Sciences and Engineering Research Council of Canada.  The second author thanks Shubham Dwivedi and Panagiotis Gianniotis for useful discussions.
The first author thanks Ash and Donna Ivey for their kind hospitality while this project was initiated.

\section{Preliminaries} \label{prelimsec}

In this section we first review some basic facts about the isometric flow of $\G$-structures and the associated soliton equations. We mostly follow the notation of~\cite{DGK}. Then we consider a particular class of Riemannian metrics and derive some general formulas for the Laplacian and the norm squared of the torsion of a $\G$-structure in this context.

\subsection{The isometric flow and isometric solitons} \label{isometricsec}

The isometric flow is 
\begin{equation} \label{isoflowphi}
\dibb{\varphi}t = \tdiv T \hk \psi
\end{equation}
where $T$ is the torsion tensor of $\varphi$ and $\psi = \hstar_{\varphi} \varphi$. Let us denote by $\phir = \varphi(0)$ the initial $\G$-structure, which we use as a reference $\G$-structure throughout, and let $[\phir]_g$ be the isometry class of $\phir$. That is, $[\phir]_g$ is the set of all $\G$-structures on $M$ inducing the same metric $g$ as $\phir$. All norms, inner products, musical isomorphisms, and Hodge star operators $\hstar$ are with respect to this fixed metric. Unless stated otherwise, we compute in a local \emph{orthonormal} frame, so all indices are subscripts and any repeated indices are summed over all values from $1$ to $7$. The symbol $\Delta$ denotes the \emph{analyst's Laplacian} $\Delta = \nabla_k \nabla_k$.

The space $[\phir]_g$ of $\G$-structures isometric to $\phir$ has an explicit parametrization, due to Bryant~\cite{Bryant}. Any element $\phii$ of $[\phir]_g$ corresponds uniquely (up to an overall sign) to a pair $(f, X)$ of a smooth function $f$ on $M$ and a smooth vector field $X$ on $M$ such that $f^2 + |X|^2 = 1$. As given in~\cite[Equation (2.16)]{DGK}, the element $\varphi(f, X) \in [\phir]_g$ is
\begin{equation} \label{isog2form}
\phii = (f^2-|X|^2) \phir -2 f X \hk \psir +2 X^{\flat} \wedge (X \hk \phir)
\end{equation}
where $\psir =\hstar \phir$.

The isometric flow~\eqref{isoflowphi} can thus be expressed as a flow of the pair $(f, X)$. This is done in~\cite[Proposition 2.8]{DGK}, where it is shown that~\eqref{isoflowphi} is equivalent to the coupled system
\begin{align}
\dot f & = \tfrac{1}{2} \langle X, \tdiv T \rangle, \nonumber \\
\dot X & = -\tfrac{1}{2} f \tdiv T - \tfrac{1}{2} X \times_0 \tdiv T, \label{isoflowX}
\end{align}
where $\times_0$ denotes the cross product relative to initial $\G$-structure $\phir = \phii(0)$, and the dot indicates time derivative under the isometric flow. Note that since $f^2 + |X|^2 = 1$, the evolution equation for $f$ follows from the evolution equation for $X$. Nevertheless the equation for $f$ is frequently useful.

The torsion $T$ of $\phii = \phii(f, X)$ given by~\eqref{isog2form} can be expressed in terms of the torsion $\overline{T}$ of the initial $\G$-structure $\overline{\phii}$. This is given in~\cite[Lemma 2.9]{DGK}, where it is established that in terms of a local orthonormal frame we have
\begin{align} \nonumber
T_{pq} & = (1 - 2|X|^2)\overline{T}_{pq} + 2 \overline{T}_{pm} X_m X_q + 2 f \overline{T}_{pm} X_\ell \phir_{m \ell q} \\ \label{torsionfX}
& \qquad + 2 X_\ell \nabla_p X_m \phir_{\ell m q} + 2 \nabla_p f X_q - 2 f \nabla_p X_q.
\end{align}
Moreover, the divergence $\tdiv T$ of $\phii$, defined to be $(\tdiv T)_k = \nabla_i T_{ik}$, is shown in~\cite[Corollary 2.10]{DGK} to be
\begin{align} \nonumber
(\tdiv T)_{q} & = (1-2|X|^2) (\tdiv \overline{T})_q - 4 X_m \nabla_p X_m \overline{T}_{pq} + 2 (\tdiv \overline{T})_m X_m X_q + 2 \overline{T}_{pm} \nabla_p X_m X_q \\ \nonumber
& \qquad + 2 \overline{T}_{pm} X_m \nabla_p X_q + 2 \nabla_p f \overline{T}_{p\ell} X_m \phir_{\ell mq} + 2 f (\tdiv \overline{T})_\ell X_m \phir_{\ell mq} + 2 f \overline{T}_{p\ell} \nabla_p X_m \phir_{\ell mq} \\ \label{divTfX}
& \qquad + 2 X_\ell \nabla_p \nabla_p X_m \phir_{\ell mq} - 2 \nabla_p X_\ell X_m \overline{T}_{ps} \overline{\psi}_{s\ell mq} + 2 \nabla_p \nabla_p f X_q - 2 f \nabla_p \nabla_p X_q.
\end{align}

In the present paper, we only consider the special case when the initial $\G$-structure $\phir$ is torsion-free. In particular, the metric is Ricci-flat, so $\nabla_p \nabla_p X_q = (\Delta X)_q$ in this case. Setting $\overline{T} = 0$ in both~\eqref{torsionfX} and~\eqref{divTfX} results in the much simpler expressions
\begin{equation} \label{simpleT}
T_{pq} = 2X_\ell \nabla_p X_m \phir_{\ell m q} +2\nabla_p f X_q -2f \nabla_p X_q
\end{equation}
and
\begin{equation} \label{simpledivT}
(\tdiv T)_q = 2 X_\ell (\Delta X)_m \phir_{\ell m q} +2 (\Delta f) X_q -2 f (\Delta X)_q
\end{equation}
for the torsion $T$, and its divergence $\tdiv T$, of the isometric $\G$-structure $\phii = \phii(f, X)$ given by~\eqref{isog2form}.

The formulas for $T$ and $\tdiv T$ simplify further for many of the $\G$-structures we consider below.
(Note that $Y \hk T$ denotes the contraction of $Y$ with $T$ on the first index, i.e., $(Y \hk T)_k = Y_p T_{pk}$.)

\begin{cor} \label{cor:T-phidropout}
(a) Suppose $Y$ is a vector field such that $\nabla_Y X$ is parallel to $X$. Then 
$$ (Y \hk T)_q = Y_p T_{pq} = 2(Yf) X_q - 2 f (\nabla_Y X)_q. $$
(b) Suppose $\Delta X$ is parallel to $X$. Then
$$ (\tdiv T)_q = 2 (\Delta f) X_q -2 f (\Delta X)_q. $$
\end{cor}
\begin{proof} Let $\nabla_Y X = \rho X$ for some smooth function $\rho$. Then using~\eqref{simpleT} we have
\begin{align*}
Y_p T_{pq} & = 2X_{\ell} (\nabla_Y X)_m \phir_{\ell m q} + 2 (\nabla_Y f) X_q - 2 f (\nabla_Y X)_q \\
& = 2 \rho X_{\ell} X_m \phir_{\ell m q} + 2 (Yf) X_q - 2 f (\nabla_Y X)_q,
\end{align*}
and the first term vanishes by skew-symmetry of $\phir_{\ell m q}$. The proof for (b) using~\eqref{simpledivT} is similar.
\end{proof}

\begin{prop} \label{prop:T-formula}
Let $\phir$ be torsion-free and let $\phii = \phii(f, X)$ be given by~\eqref{isog2form}. Then the torsion $T$ of $\phii$ satisfies
\begin{equation} \label{normT}
\tfrac{1}{4} |T|^{2} = |\nabla X|^2 + |\nabla f|^2.
\end{equation}
\end{prop}
\begin{proof}
We first observe by differentiating $f^2 +|X|^2=1$ that
\begin{equation} \label{fdf}
f \nabla_{p} f = -X_{q}\nabla_{p} X_{q}.
\end{equation}
Taking the norm squared of both sides of~\eqref{fdf} gives
\begin{equation} \label{fdf-squared}
f^2 |\nabla f|^2 = X_{k}\nabla_{p} X_{k} X_{q}\nabla_{p} X_{q}.
\end{equation}
Using~\eqref{simpleT}, we compute
\begin{align*}
\tfrac{1}{4} |T|^{2}= \tfrac{1}{4} T_{pq}T_{pq} &= X_\ell \nabla_p X_m \phir_{\ell m q} X_k \nabla_p X_n \phir_{k n q} + |\nabla f|^2 |X|^2 +f^2 |\nabla X|^2 \\
&\qquad+ 2X_\ell \nabla_p X_m \phir_{\ell m q} \nabla_p f X_q -2 f X_\ell \nabla_p X_m \phir_{\ell m q}\nabla_p X_q -2 f \nabla_p f X_q \nabla_p X_q.
\end{align*}
The last two terms involving $\phir$ are zero because the 3-form components are summed against terms symmetric in $\ell,q$ and $m,q$ respectively. Using the identity~\cite[Lemma A.12]{K-flows} given by
$$ \phir_{\ell m q} \phir_{k n q} = \delta_{\ell k} \delta_{mn} - \delta_{\ell n} \delta_{mk} - \psir_{\ell m k n} $$
to expand the first term, and using~\eqref{fdf} to eliminate derivatives of $f$ from the last term, we obtain
\begin{align*}
\tfrac{1}{4} |T|^{2} & = X_k X_\ell \nabla_p X_m \nabla_p X_n (\delta_{\ell k} \delta_{mn} - \delta_{\ell n} \delta_{mk} - \psir_{\ell m k n}) \\
&\qquad {} + |\nabla f|^2 |X|^2 +f^2 |\nabla X|^2 + 2 X_k \nabla_p X_k X_q \nabla_p X_q \\
& = |X|^2 |\nabla X|^2 - X_{k} X_{\ell} \nabla_{p}X_{k}\nabla_{p}X_{\ell}+|\nabla f|^2 |X|^2 + f^2 |\nabla X|^2 + 2 X_k \nabla_p X_k X_q \nabla_p X_q.
\end{align*}
Using the constraint $f^2 + |X|^2 = 1$ twice, and the identity~\eqref{fdf-squared}, the above becomes
\begin{align*}
\tfrac{1}{4} |T|^{2} & = |\nabla X|^2 (|X|^2 + f^2) + |X|^2|\nabla f|^2 + X_k \nabla_p X_k X_\ell \nabla_p X_\ell \\
& = |\nabla X|^2 + (1-f^2) |\nabla f|^2 + X_k \nabla_p X_k X_\ell \nabla_p X_\ell \\
& = |\nabla X|^2 + |\nabla f|^2
\end{align*}
as claimed.
\end{proof}

A \emph{soliton} for the isometric flow~\eqref{isoflowphi} is a self-similar solution (a solution which evolves by scalings and diffeomorphisms). It is shown in~\cite[Section 2.5]{DGK} that the data of an isometric flow soliton is encoded by a triple $(\varphi, Y, c)$ where $Y$ is a vector field and $c$ a constant such that
\begin{equation} \label{solitoncond}
\begin{aligned}
\Lie_Y g & = 2c \, g, \\
\tdiv T & = - \tfrac{1}{2} \curl_{\varphi} Y +Y \hk T,
\end{aligned}
\end{equation}
where $g$ is the metric and $T$ the torsion tensor associated to $\varphi$. Here $\curl_{\phii}$ denotes the \emph{curl} operator on vector fields corresponding to the $\G$-structure $\phii$, which can be expressed invariantly as
\begin{equation} \label{curldefn1}
\curl_{\varphi} Y = \left[ \hstar (d Y^{\flat}) \wedge \psi\right]^{\sharp}.
\end{equation}
Alternatively, in a local orthonormal frame we have
\begin{equation} \label{curldefn}
(\curl_{\varphi} Y)_k = \nabla_i Y_j \varphi_{ijk}.
\end{equation}

\begin{rmk} \label{rmk:soliton-notation}
In~\eqref{solitoncond} we have changed some notation from~\cite[Section 2.5]{DGK}. Namely, we have changed their $X$ to $Y$ because we use $X$ in the description $(f,X)$ of a member of the isometry class $[\phir]_g$, and we have also replaced their $c$ by $2c$ to avoid many factors of $\frac{1}{2}$ in our subsequent computations.
\end{rmk}

The isometric soliton $(\varphi, Y, c)$ is called \emph{shrinking}, \emph{steady}, or \emph{expanding} if $c$ is positive, zero, or negative, respectively. This nomenclature is chosen in analogy with solitons for the Ricci flow. Note that despite this, the associated self-similar solution $\phii(t)$ of the isometric flow induces the same Riemannian metric $g$ for all time, so the metric is \emph{not} changing. The action of the scaling of $\phii$ and the action by diffeomorphisms combine in such a way to ensure that the metric $g$ of $\phii(t)$ is independent of $t$, although the $\G$-structure $\phii(t)$ definitely does evolve in time. The precise correspondence between the triple $(\varphi, Y, c)$ and the self-similar solution of the isometric flow is explained in~\cite[Lemma 2.17]{DGK}, but we do not have need for it here.

Note also that since the metric $g$ is fixed, we do not have the freedom to scale $\varphi$. It then follows from~\eqref{solitoncond} that we cannot scale $Y$ or $c$ either. This is in contrast to solitons for the Ricci flow, where one can always scale the metric to ensure $c \in \{ -1, 0 , 1 \}$.

Because we fix the class $[\phir]_g$ of $\G$-structures inducing the same Riemannian metric $g$, we can describe $\phii$ by~\eqref{isog2form} in terms of the pair $(f, X)$. Thus, the data of an isometric soliton consists of the quadruple $(f, X, Y, c)$ where $(f, X)$ is determined up to an overall sign and satisfies $f^2 + |X|^2 = 1$.

\subsection{Formulas for a particular class of Riemannian metrics} \label{Laplacian-sec}

In this section we consider the total space of a real vector bundle as a smooth manifold, equipped with a special class of Riemannian metrics. This situation includes as particular cases all the torsion-free $\G$-manifolds we consider in this paper, namely cylinders over Calabi-Yau manifolds, cones over nearly K\"ahler manifolds (including the Euclidean space $\R^7$ over the round sphere $S^6$), and the Bryant--Salamon $\G$-manifolds. The induced Riemannian metrics in these examples can all be considered as particular cases of this special class of metrics. Treating them all in this way allows us to derive general formulas for certain Laplacians and other geometric quantities which we can then apply to all of these examples.

Let $E$ be a smooth manifold which is the total space of a real vector bundle over a base $B$, with projection $\pi \colon E \to B$. We assume that $B$ is equipped with a Riemannian metric $\gdn$, and that $E$ is equipped with a fibre metric and a compatible connection. Hence $TE$ splits as a sum $H \oplus V$ of horizontal and vertical sub-bundles, where $H \cong \pi^* TB$ and $V = \ker d\pi$, and thus $H$ and $V$ inherit fibre metrics. We thus get an induced splitting $T^* E = H^* \oplus V^*$.

Let $n = \dim B$ and $m = \rank E$. Let $\{ b_1, \ldots, b_n \}$ be a local $\gdn$-orthonormal coframe for $B$, defined on $\scrU \subset B$. The pullbacks to $E$, still denoted $\{ b_1, \ldots, b_n \}$, are a local basis of sections for the bundle $H^*$ of horizontal $1$-forms on $E$. Let $\{ \sigma_1, \ldots, \sigma_m \}$ be a local orthonormal frame of sections of $E$ also defined on $\scrU$. These determine local \emph{fibre coordinates} $x_{\alpha}$ on $\pi^{-1}(\scrU) \subset E$, where $1 \leq \alpha \leq m$. They also determine 1-forms $\{ \zeta_1, \ldots, \zeta_m \}$ which annihilate the distribution $H$ and satisfy $\zeta_\alpha(\sv) = x_\alpha(\rj \sv)$ for $\sv\in V$, where $\rj$ denotes the canonical isomorphism between the fibers $V_p$ and $E_{\pi(p)}$ for $p\in E$. In particular, the $\zeta_\alpha$ are a local frame of sections for the bundle $V^*$ of vertical $1$-forms on $E$. It is well-known (and easy to check) that
\begin{equation} \label{eq:general-zeta}
\zeta_{\alpha} = dx_{\alpha} + x_{\beta} \omega_{\alpha \beta},
\end{equation}
where the $\omega_{\alpha \beta}$ are $1$-forms pulled back from the base $B$, and are skew-symmetric in $\alpha, \beta$.

Let $r$ denote the distance from the origin in the fibres of $E$ with respect to the fibre metric. In terms of these fibre coordinates, which are defined with respect to a local \emph{orthonormal} frame for $E$, we have
\begin{equation} \label{eq:general-rdr}
r^2 = \sum_{\alpha=1}^m x_{\alpha}^2 \quad \text{and} \quad r dr = x_{\alpha} d x_{\alpha} = x_{\alpha} \zeta_{\alpha}.
\end{equation}

Fix two smooth positive functions $h, k$ of $r \in [0, \infty)$. We define a Riemannian metric $g$ on the total space of $E$ such that $\{ h b_1, \ldots, h b_n, k \zeta_1, \ldots, k \zeta_m \}$ is a local orthonormal coframe. That is,
\begin{equation} \label{eq:general-bundle-metric}
g = h^2 \Big(\sum_{i=1}^n b_i^2\Big) + k^2 \Big(\sum_{\alpha=1}^m \zeta_{\alpha}^2\Big). 
\end{equation}
This is clearly independent of the choices of $\{ b_1, \ldots, b_n \}$ and $\{ \zeta_1, \ldots, \zeta_m \}$.

Let $\Delta$ denote the rough Laplacian for this metric~\eqref{eq:general-bundle-metric}. The goal of this section is to compute explicit formulas, in terms of $h, k$ and their derivatives, for $\Delta s$, $\Delta X$, and $|\nabla X|^2$, when $s = s(r)$ is any smooth function on $E$ depending only on $r \geq 0$ and $X$ is a vector field of the form $X = s \frac{\partial}{\partial r}$. We need such formulas in order to compute $\tdiv T$ in equation~\eqref{simpledivT} when $f = f(r)$ and $X = s(r) \frac{\partial}{\partial r}$.

In what follows, repeated indices $i, j$ are summed from $1$ to $n$, and repeated indices $\alpha, \beta$ are summed from $1$ to $m$, regardless of whether they appear as superscripts or subscripts.

Note that the general formula for the rough Laplacian from~\eqref{eq:rough-Lap} holds with respect to an arbitrary local frame for $TE$. We choose the \emph{particular} local frame to be $\{ \se_1, \ldots, \se_n, \sv_1, \ldots, \sv_m \}$ where
\begin{equation} \label{eq:general-frame}
b_i (\se_j) = \delta_{ij}, \quad b_i (\sv_{\beta}) = 0, \quad \zeta_{\alpha} (\se_j) = 0, \quad \zeta_{\alpha} (\sv_{\beta}) = \delta_{\alpha \beta}.
\end{equation}
It follows from~\eqref{eq:general-bundle-metric} that
\begin{equation} \label{eq:bundle-metric-frame}
g(\se_i, \se_j) = h^2 \delta_{ij}, \quad g(\se_i, \sv_{\alpha}) = 0, \quad g(\sv_{\alpha}, \sv_{\beta}) = k^2 \delta_{\alpha \beta}. 
\end{equation}
Note that the vector fields $\sv_1, \ldots, \sv_m$ are vertical for the projection $\pi \colon E \to B$. 

From~\eqref{eq:general-zeta} and~\eqref{eq:general-frame}, we have
$$ \sv_{\beta} x_{\alpha} = dx_{\alpha} (\sv_{\beta}) = \zeta_{\alpha} (\sv_{\beta}) = \delta_{\alpha \beta}, $$
so $\sv_{\alpha}$ can be thought of as partial differentiation with respect to $x_{\alpha}$ in the vertical (fibre) direction. In particular, since~\eqref{eq:general-rdr} then gives $dr (\sv_{\alpha}) = r^{-1} x_{\alpha}$, we have
\begin{equation} \label{eq:sv-r}
\sv_{\alpha} r = \frac{x_{\alpha}}{r}.
\end{equation}
It follows that partial differentiation with respect to $r$ in the vertical direction is given by the vertical vector field
\begin{equation} \label{eq:dbydr}
\dib{r} = (\sv_{\alpha} r) \sv_{\alpha} = \frac{x_{\alpha}}{r} \sv_{\alpha}.
\end{equation}
The above also shows that
\begin{equation} \label{eq:normdbydr}
| \dbydr |^2 = k^2.
\end{equation}
We further obtain the relations
\begin{align*}
\langle \dbydr, \sv_{\beta} \rangle & = \frac{x_{\alpha}}{r} \langle \sv_{\alpha}, \sv_{\beta} \rangle = \frac{k^2 x_{\alpha}}{r} \delta_{\alpha \beta}, & \langle \dbydr, \se_i \rangle & = 0,
\end{align*}
from which it follows that
\begin{equation} \label{eq:gradr}
\langle \dbydr, X \rangle = k^2 dr (X) = k^2 \langle \nabla r, X \rangle \quad \text{for all $X$, so} \quad \nabla r = \frac{1}{k^2} \dib{r}.
\end{equation}

For $s = s(r)$, by~\eqref{eq:general-rdr} we have
$$ ds = s' dr = \frac{s'}{r} r dr = \frac{s'}{r} x_{\beta} \zeta_{\beta}, $$
and thus
\begin{equation} \label{eq:nabis}
\nabla_i s : = \se_i (s) = ds (\se_i) = 0,
\end{equation}
and
\begin{equation} \label{eq:nabalphas}
\nabla_{\alpha} s : = \sv_{\alpha} (s) = ds (\sv_{\alpha}) = \frac{x_{\alpha} s'}{r}.
\end{equation}
It also follows immediately from the above or from~\eqref{eq:gradr} that
\begin{equation} \label{eq:nablas}
\nabla s = g^{\alpha \beta} \frac{x_{\alpha}}{r} s' \sv_{\beta} = \frac{s'}{k^2} \dib{r}.
\end{equation}
We now derive various preliminary results for metrics of the form~\eqref{eq:general-bundle-metric}.

\begin{lemma} \label{lemma:LYg}
Let $g$ be the metric of~\eqref{eq:general-bundle-metric}, and let $Y = b \frac{\partial}{\partial r}$ where $b = b(r)$. Suppose that $\mathcal{L}_Y g = 2 c g$.
Then $c h = b h'$ and $c k = b k' + k b'$. Moreover, if $m>1$ then $b = \mu r$ for some constant $\mu$, so that 
$c h = \mu r h'$ and $c k = \mu r k' + \mu k$.
\end{lemma}
\begin{proof}
From~\eqref{eq:general-zeta} and~\eqref{eq:dbydr}, we compute that
$$ Y \hk \zeta_{\alpha} = \frac{b x_{\gamma}}{r} \sv_{\gamma} \hk ( d x_{\alpha} + x_{\beta} \omega_{\alpha \beta}) = \frac{b x_{\gamma}}{r} (\delta_{\gamma \alpha} + 0 ) = \frac{b x_{\alpha}}{r}, $$
and
$$ Y \hk (d \zeta_{\alpha}) = \frac{b x_{\gamma}}{r} \sv_{\gamma} \hk ( d x_{\beta} \wedge \omega_{\alpha \beta} + x_{\beta} d \omega_{\alpha \beta} ) = \frac{b x_{\gamma}}{r} ( \delta_{\gamma \beta} \omega_{\alpha \beta} + 0 ) = \frac{b x_{\beta}}{r} \omega_{\alpha \beta}. $$
Using these, we obtain
\begin{align*}
\mathcal{L}_Y \zeta_{\alpha} & = Y \hk d \zeta_{\alpha} + d (Y \hk \zeta_{\alpha}) = \frac{b x_{\beta}}{r} \omega_{\alpha \beta} + d \Big( \frac{b x_{\alpha}}{r} \Big) \\
& = \frac{b x_{\beta}}{r} \omega_{\alpha \beta} + \Big( \frac{b'}{r} - \frac{b}{r^2} \Big) x_{\alpha} dr + \frac{b}{r} dx_{\alpha} \\
& = \frac{b}{r} \zeta_{\alpha} + \Big( \frac{b'}{r} - \frac{b}{r^2} \Big) x_{\alpha} dr.
\end{align*}
From the above and the fact that $r dr = x_{\alpha} \zeta_{\alpha}$, we obtain
$$ \mathcal{L}_Y \left( \sum_{\alpha=1}^m \zeta_{\alpha}^2 \right) = \frac{2b}{r} \left( \sum_{\alpha=1}^m \zeta_{\alpha}^2 \right) + 2 \Big( b' - \frac{b}{r} \Big) dr^2. $$
Since $\mathcal{L}_Y b_i = 0$ for $Y = b \frac{\partial}{\partial r}$, from the above and~\eqref{eq:general-bundle-metric} we deduce that
\begin{equation} \label{eq:LYg-general}
\mathcal{L}_Y g = 2 b h h' \left( \sum_{i=1}^n b_i^2 \right) + \Big( 2 b k k' + \frac{2 b k^2}{r} \Big) \left( \sum_{\alpha=1}^m \zeta_{\alpha}^2 \right) + 2 k^2 \Big( b' - \frac{b}{r} \Big) dr^2.
\end{equation}

Suppose $m=1$. Then $\alpha = 1$, $x_1 = r$, and $\zeta_1 = d x_1 = dr$. Thus $\sum_{\alpha} \zeta_{\alpha}^2 = dr^2$, and using~\eqref{eq:LYg-general}, the equation $\mathcal{L}_Y g = 2 c g$ becomes
$$ 2 b h h' \left( \sum_{i=1}^n b_i^2 \right) + \left( 2 b k k' + \frac{2 b k^2}{r} + 2 k^2 \Big( b' - \frac{b}{r}\Big) \right) dr^2 = 2 c \left( h^2 \left( \sum_{i=1}^n b_i^2 \right) + k^2 dr^2 \right). $$
Simplifying and comparing coefficients yields $2 b h h' = 2 c h^2$ and $2 b k k' + 2 k^2 b' = 2 c k^2$. Since $h, k > 0$, this becomes $ch = b h'$ and $c k  = b k' + k b'$ as claimed.

Now suppose that $m > 1$. The $dr^2$ term in~\eqref{eq:LYg-general} will contribute factors of the form $\frac{x_{\alpha} x_{\beta}}{r^2} \zeta_{\alpha} \zeta_{\beta}$ for $\alpha \neq \beta$, which have no counterparts in $2 c g$. Thus we must have $b' - \frac{b}{r} = 0$, so $b = \mu r$ for some constant $\mu$. Then equating coefficients of $\sum_i b_i^2$ and $\sum_{\alpha} \zeta_{\alpha}^2$ again gives $ch = b h'$ and $ck = b k' + k b'$, which yields the result upon substituting $b = \mu r$.
\end{proof}

\begin{lemma} \label{lemma:bundle-Gammas}
Let $g$ be the metric of~\eqref{eq:general-bundle-metric}, and let $\{ \se_1, \ldots, \se_n, \sv_1, \ldots, \sv_m \}$ be the local frame given by~\eqref{eq:general-frame}. Then the Lie brackets $[\se_i, \sv_{\gamma}]$ and $[\sv_{\alpha}, \sv_{\beta}]$ satisfy
\begin{equation} \label{eq:bundle-Gammas-brackets}
g( [\se_i, \sv_{\gamma}], \se_j ) = 0, \qquad g( [\sv_{\alpha}, \sv_{\beta}], \sv_{\gamma} ) = 0.
\end{equation}
\end{lemma}
\begin{proof}
We compute using~\eqref{eq:general-frame} that
\begin{align*}
(d b_j) (\se_i, \sv_{\gamma}) & = \se_i \big( b_j(\sv_{\gamma}) \big) - \sv_{\gamma} \big( b_j (\se_i) \big) - b_j ([\se_i, \sv_{\gamma}]) \\
& = 0 - 0 - b_j ([\se_i, \sv_{\gamma}]).
\end{align*}
Thus $[\se_i, \sv_{\gamma}]$ is in the span of $\sv_1, \ldots, \sv_m$ if and only if $(db_j)(\se_i, \sv_{\gamma}) = 0$. But $b_j$ is (pulled back from) a $1$-form on the base $B$, and thus so is $d b_j$, hence $(db_j)(\se_i, \sv_{\gamma}) = 0$, proving the first statement.

Similarly, we have
\begin{align*}
(d \zeta_{\gamma}) (\sv_{\alpha}, \sv_{\beta}) & = \sv_{\alpha} \big( \zeta_{\gamma}(\sv_{\beta}) \big) - \sv_{\beta} \big( \zeta_{\gamma} (\sv_{\alpha}) \big) - \zeta_{\gamma} ([\sv_{\alpha}, \sv_{\beta}]) \\
& = 0 - 0 - \zeta_{\gamma} ([\sv_{\alpha}, \sv_{\beta}]).
\end{align*}
Thus $[\sv_{\alpha}, \sv_{\beta}]$ is in the span of $\se_1, \ldots, \se_n$ if and only if $(d\zeta_{\gamma})(\sv_{\alpha}, \sv_{\beta}) = 0$. Taking $d$ of~\eqref{eq:general-zeta} shows that $d \zeta_{\gamma} = dx_{\beta} \wedge \omega_{\gamma\beta} + x_{\beta} d \omega_{\gamma\beta}$, from which it is clear, since $\omega_{\gamma\beta}$ is pulled back from the base, that $(d \zeta_{\gamma})(\sv_{\alpha}, \sv_{\beta}) = 0$, proving the second statement.
\end{proof}

\begin{prop} \label{prop:bundle-Gammas}
Let $g$ be the metric of~\eqref{eq:general-bundle-metric}, and let $\{ \se_1, \ldots, \se_n, \sv_1, \ldots, \sv_m \}$ be the local frame given by~\eqref{eq:general-frame}. Write $g_{ij} = g(\se_i, \se_j) = h^2 \delta_{ij}$ and $g_{\alpha \beta} = g(\sv_{\alpha}, \sv_{\beta}) = k^2 \delta_{\alpha \beta}$, with inverses $g^{ij} = h^{-2} \delta^{ij}$ and $g^{\alpha \beta} = k^{-2} \delta^{\alpha \beta}$, respectively. Define the Christoffel symbols $\Gamma$ of the Levi-Civita connection $\nabla$ of $g$ by
$$ \nabla_{\se_i} \se_j = \Gamma_{ij}^k \se_k + \Gamma_{ij}^{\gamma} \sv_{\gamma}, \qquad \nabla_{\sv_{\alpha}} \sv_{\beta} = \Gamma_{\alpha \beta}^k \se_k + \Gamma_{\alpha \beta}^{\gamma} \sv_{\gamma}, $$
and similarly for other covariant derivatives, which we do not require. Then we have
\begin{equation} \label{eq:bundle-Gammas}
g^{ij} \Gamma^{\gamma}_{ij} = - \frac{n h'}{h k^2 r} x_{\gamma}, \qquad g^{\alpha \beta} \Gamma_{\alpha \beta}^{\gamma} = \frac{(2-m)k'}{k^3 r} x_{\gamma},
\end{equation}
and
\begin{equation} \label{eq:bundle-Gammas-v}
\Gamma_{\alpha \beta}^{\gamma} = \frac{k'}{kr} ( x_{\alpha} \delta_{\beta \gamma} + x_{\beta} \delta_{\alpha \gamma} - x_{\gamma} \delta_{\alpha \beta} ).
\end{equation}
\end{prop}
\begin{proof}
From $\nabla_{\se_i} \se_j = \Gamma_{ij}^k \se_k + \Gamma_{ij}^{\gamma} \sv_{\gamma}$ and~\eqref{eq:bundle-metric-frame}, we observe using metric-compatibility and torsion-freeness of $\nabla$, and the orthogonality $g( \se_i, \sv_{\gamma} ) = 0$, that
\begin{align*}
k^2 \Gamma_{ij}^{\gamma} & = g( \nabla_{\se_i} \se_j, \sv_{\gamma} ) \\
& = \se_i \big( g(\se_j, \sv_{\gamma}) \big) - g( \se_j, \nabla_{\se_i} \sv_{\gamma} ) \\
& = 0 - g( \se_j, \nabla_{\sv_{\gamma}} \se_i + [\se_i, \sv_{\gamma}] ).
\end{align*}
The term $g( \se_j, [\se_i, \sv_{\gamma}] )$ vanishes by~\eqref{eq:bundle-Gammas-brackets}. Thus using~\eqref{eq:sv-r} the above becomes
\begin{align*}
g^{ij} \Gamma^{\gamma}_{ij} & = h^{-2} \delta^{ij} k^{-2} ( - g( \se_j, \nabla_{\sv_{\gamma}} \se_i ) ) \\
& = - \frac{1}{h^2 k^2} g( \se_i, \nabla_{\sv_{\gamma}} \se_i ) = - \frac{1}{2 h^2 k^2} v_{\gamma} (g( \se_i, \se_i )) \\
& = - \frac{1}{2 h^2 k^2} \sv_{\gamma} (n h^2) = - \frac{2 n h h'}{2 h^2 k^2} \frac{x_{\gamma}}{r},
\end{align*}
establishing the first equation in~\eqref{eq:bundle-Gammas}.

From $\nabla_{\sv_{\alpha}} \sv_{\beta} = \Gamma_{\alpha \beta}^k \se_k + \Gamma_{\alpha \beta}^{\gamma} \sv_{\gamma}$ and~\eqref{eq:bundle-metric-frame}, we observe that
\begin{equation} \label{eq:bundle-Gammas-temp}
k^2 \Gamma_{\alpha \beta}^{\gamma} = g( \nabla_{\sv_{\alpha}} \sv_{\beta}, \sv_{\gamma} ).
\end{equation}
Recall the Koszul formula for the Levi-Civita connection of $g$, which is
\begin{equation*}
\begin{aligned}
2 g( \nabla_X Y, Z) & = X \big( g(Y, Z) \big) + Y \big( g(X, Z) \big) - Z \big( g(X, Y) \big) \\
& \qquad {} - g([X, Z], Y) - g([Y, Z], X) + g([X, Y], Z).
\end{aligned}
\end{equation*}
By~\eqref{eq:bundle-Gammas-brackets},~\eqref{eq:bundle-Gammas-temp}, and~\eqref{eq:sv-r}, the Koszul formula gives
\begin{align*}
2 k^2 \Gamma_{\alpha \beta}^{\gamma} = 2 g( \nabla_{\sv_{\alpha}} \sv_{\beta}, \sv_{\gamma} ) & = \sv_{\alpha} (k^2 \delta_{\beta \gamma}) + \sv_{\beta} (k^2 \delta_{\alpha \gamma}) - \sv_{\gamma} (k^2 \delta_{\alpha \beta}) - 0 - 0 + 0 \\
& = 2 k k' \Big( \frac{x_{\alpha} \delta_{\beta \gamma}}{r} + \frac{x_{\beta} \delta_{\alpha \gamma}}{r} - \frac{x_{\gamma} \delta_{\alpha \beta}}{r} \Big).
\end{align*}
Equation~\eqref{eq:bundle-Gammas-v} then follows.
Taking the trace over $\alpha, \beta$, we then obtain
\begin{align*}
g^{\alpha \beta} \Gamma_{\alpha \beta}^{\gamma} & = k^{-2} \delta^{\alpha \beta} \frac{k'}{k r} (x_{\alpha} \delta_{\beta \gamma} + x_{\beta} \delta_{\alpha \gamma} - x_{\gamma} \delta_{\alpha \beta}) \\
& = \frac{k'}{k^3 r} (x_{\gamma} + x_{\gamma} - m x_{\gamma}),
\end{align*}
establishing the second equation in~\eqref{eq:bundle-Gammas}.
\end{proof}

\begin{cor} \label{cor:general-Laplacian}
Let $s = s(r)$. Then its Laplacian $\Delta s = g^{pq} \nabla_p \nabla_q s$ with respect to the metric $g$ of~\eqref{eq:general-bundle-metric} is
\begin{equation} \label{eq:general-Laplacian}
\Delta s = \frac{1}{k^2} s'' + \Big( \frac{m-1}{k^2 r} + \frac{n h'}{h k^2} + \frac{(m-2)k'}{k^3} \Big) s'.
\end{equation}
\end{cor}
\begin{proof}
Consider the local frame $\{ \se_1, \ldots, \se_n, \sv_1, \ldots, \sv_m \}$ given by~\eqref{eq:general-frame}. Use $a,b,c$ to denote indices from $1$ to $n+m$, where the first $n$ correspond to the vectors $\se_1, \ldots, \se_n$, and the last $m$ correspond to the vectors $\sv_1, \ldots, \sv_m$. We compute using~\eqref{eq:sv-r},~\eqref{eq:nabis}, and~\eqref{eq:nabalphas} that
\begin{align*}
\Delta s & = g^{ab} \nabla_a (\nabla_b s) - g^{ab} \Gamma_{ab}^c \nabla_c s \\
& = g^{\alpha \beta} \nabla_{\alpha} (\nabla_{\beta} s) - g^{ij} \Gamma_{ij}^{\gamma} \nabla_{\gamma} s - g^{\alpha \beta} \Gamma_{\alpha \beta}^{\gamma} \nabla_{\gamma} s \\
& = \frac{1}{k^2} \delta^{\alpha \beta} \nabla_{\alpha} \Big( \frac{x_{\beta} s'}{r} \Big) - (g^{ij} \Gamma_{ij}^{\gamma} + g^{\alpha \beta} \Gamma_{\alpha \beta}^{\gamma}) \Big( \frac{x_{\gamma} s'}{r} \Big).
\end{align*}
Noting that $\nabla_{\alpha} x_{\beta} = \sv_{\alpha} x_{\beta} = \delta_{\alpha \beta}$, if we expand the first term and use Proposition~\ref{prop:bundle-Gammas} on the second term, we obtain
\begin{align*}
\Delta s & = \frac{1}{k^2} \delta^{\alpha \beta} \Big( \delta_{\alpha \beta} \frac{s'}{r} + \frac{x_{\beta} s''}{r} \frac{x_{\alpha}}{r} - \frac{x_{\beta} s'}{r^2} \frac{x_{\alpha}}{r} \Big) - \Big( - \frac{n h'}{h k^2 r} x_{\gamma} + \frac{(2-m)k'}{k^3 r} x_{\gamma} \Big) \Big( \frac{x_{\gamma} s'}{r} \Big) \\
& = \frac{1}{k^2} \Big( \frac{ms'}{r} + s'' - \frac{s'}{r} \Big) - \Big( - \frac{n h' s'}{h k^2} + \frac{(2-m)k's'}{k^3} \Big),
\end{align*}
which simplifies to~\eqref{eq:general-Laplacian}.
\end{proof}

In order to compute the Laplacian $\Delta X$ for a vector field of the form $X = s(r) \frac{\partial}{\partial r}$, we need some more preliminary results.
\begin{lemma} \label{lemma:Lap-X-prelim}
With respect to the metric~\eqref{eq:general-bundle-metric}, the following relations hold:
$$ \langle \nr \dbydr, \se_i \rangle = 0, \qquad \langle \nr \dbydr, \sv_{\beta} \rangle = \frac{k k'}{r} x_{\beta}. $$
\end{lemma}
\begin{proof}
From $d (dr) = 0$ and~\eqref{eq:gradr}, we get
\begin{align*}
0 & = (d(dr))(\se_i, \dbydr) = \se_i \big( dr(\dbydr) \big) - \dbydr \big( dr (\se_i) \big) - (dr) ( [\se_i, \dbydr] ) \\
& = \se_i (1) - 0 - \tfrac{1}{k^2} \langle \dbydr, [\se_i, \dbydr] \rangle,
\end{align*}
yielding
\begin{equation} \label{eq:lemma:Lap-X-temp}
\langle \dbydr, [ \dbydr, \se_i ] \rangle = 0.
\end{equation}
Now we observe that
\begin{align*}
\langle \nr \dbydr, \se_i \rangle & = \dbydr \big( \langle \dbydr, \se_i \rangle \big) - \langle \dbydr, \nr \se_i \rangle \\
& = 0 - \langle \dbydr, \nabla_{\se_i} \dbydr + [ \dbydr, \se_i ] \rangle \\
& = - \tfrac{1}{2} \se_i (|\dbydr|^2) - \langle \dbydr, [ \dbydr, \se_i ] \rangle.
\end{align*}
The second term vanishes by~\eqref{eq:lemma:Lap-X-temp}, and the first term vanishes by~\eqref{eq:normdbydr} and~\eqref{eq:nabis}, since $\se_i (k^2) = 0$.

Next we compute
\begin{align*}
\langle \nr \dbydr, \sv_{\beta} \rangle & = \left \langle \nabla_{\tfrac{x_{\alpha}}{r} \sv_{\alpha}} \Big( \frac{x_{\gamma}}{r} \sv_{\gamma} \Big), \sv_{\beta} \right \rangle = \frac{x_{\alpha}}{r} \left \langle \nabla_{\sv_{\alpha}} \Big( \frac{x_{\gamma}}{r} \sv_{\gamma} \Big), \sv_{\beta} \right \rangle \\
& = \frac{x_{\alpha}}{r} \left \langle \frac{\delta_{\alpha \gamma}}{r} \sv_{\gamma} - \frac{x_{\gamma}}{r^2} \frac{x_{\alpha}}{r} \sv_{\gamma} + \frac{x_{\gamma}}{r} \nabla_{\sv_{\alpha}} \sv_{\gamma}, \sv_{\beta} \right \rangle \\
& = \left \langle \frac{x_{\gamma}}{r^2} \sv_{\gamma} - \frac{x_{\gamma}}{r^2} \sv_{\gamma} + \frac{x_{\alpha} x_{\gamma}}{r^2} \Gamma_{\alpha \gamma}^{\delta} \sv_{\delta}, \sv_{\beta} \right \rangle \\
& = \frac{k^2 x_{\alpha} x_{\gamma}}{r^2} \Gamma^{\beta}_{\alpha \gamma}.
\end{align*}
Substituting~\eqref{eq:bundle-Gammas-v} above yields
\begin{align*}
\langle \nr \dbydr, \sv_{\beta} \rangle & = \frac{k^2 x_{\alpha} x_{\gamma}}{r^2} \frac{k'}{kr} ( x_{\alpha} \delta_{\beta \gamma} + x_{\gamma} \delta_{\alpha \beta} - x_{\beta} \delta_{\alpha \gamma} ) \\
& = \frac{k k'}{r^3} (r^2 x_{\beta} + r^2 x_{\beta} - r^2 x_{\beta}) = \frac{kk'}{r} x_{\beta}
\end{align*}
as claimed.
\end{proof}

\begin{cor} \label{cor:nrr}
With respect to the metric~\eqref{eq:general-bundle-metric}, we have
$$ \nabla_{\dbydr} \dib{r} = \frac{k'}{k} \dib{r}. $$
\end{cor}
\begin{proof}
We can write
$$ \nabla_{\dbydr} \dib{r} = c_i \se_i + c_{\alpha} \sv_{\alpha}. $$
Lemma~\ref{lemma:Lap-X-prelim} shows that $c_i = 0$ and $k^2 c_{\alpha} = \frac{k k'}{r} x_{\alpha}$, so $c_{\alpha} = \frac{k'}{kr} x_{\alpha}$. Thus we have
$$ \nabla_{\dbydr} \dib{r} = \frac{k'}{kr} x_{\alpha} \sv_{\alpha} = \frac{k'}{k} \dib{r}. $$
as claimed.
\end{proof}

For future use we note that Corollary~\ref{cor:nrr} implies that
\begin{equation} \label{eq:radialXderiv}
\text{for $X = s(r) \dbydr$, \, we have} \qquad \nr X = \Big( s'+ s \dfrac{k'}{k} \Big) \dib{r}.
\end{equation}

\begin{lemma} \label{lemma:Lap-RF}
If we further assume that the metric~\eqref{eq:general-bundle-metric} is \emph{Ricci-flat}, then we have
$$ \Delta \Big( \frac{1}{k^2} \dib{r} \Big) = \Big( - \frac{2(m-1) k'}{k^5 r} - \frac{(m-1)}{k^4 r^2} + \frac{n h''}{h k^4} - \frac{n (h')^2}{h^2 k^4} - \frac{2 n h' k'}{h k^5} + \frac{(m-2) k''}{k^5} - \frac{3(m-2)(k')^2}{k^6} \Big) \dib{r}. $$
\end{lemma}
\begin{proof}
Using the Ricci identity, for any function $F$ we have
\begin{align*}
\Delta (\nabla_i F) & = \nabla_p \nabla_p \nabla_i F = \nabla_p \nabla_i \nabla_p F = \nabla_i \nabla_p \nabla_p F - R_{pipm} \nabla_m F \\
& = \nabla_i (\Delta F) + R_{im} \nabla_m F = \nabla_i (\Delta F).
\end{align*}
(Up to the musical isomorphism determined by the metric, the above is just an explicit demonstration of the fact that for a Ricci-flat metric, the Weitzenb\"ock formula shows that the Hodge Laplacian $\Delta_d$ on $1$-forms agrees with the rough Laplacian $\Delta$, and thus $\Delta dF = \Delta_d dF = d \Delta_d F = d \Delta F$.)

Thus, using~\eqref{eq:gradr},~\eqref{eq:general-Laplacian} with $s(r) = r$, and~\eqref{eq:nablas} we compute
\begin{align*}
\Delta \Big( \frac{1}{k^2} \dib{r} \Big) & = \Delta ( \nabla r ) = \nabla (\Delta r) \\
& = \nabla \Big( \frac{m-1}{k^2 r} + \frac{n h'}{h k^2} + \frac{(m-2)k'}{k^3} \Big) \\
& = \frac{1}{k^2} \Big( \frac{m-1}{k^2 r} + \frac{n h'}{h k^2} + \frac{(m-2)k'}{k^3} \Big)' \dib{r}
\end{align*}
Expanding the derivative above yields the result.
\end{proof}

\begin{cor} \label{cor:general-Laplacian-1}
Assume again that the metric~\eqref{eq:general-bundle-metric} is Ricci-flat. Let $X = s(r) \frac{\partial}{\partial r}$. Then its Laplacian $\Delta X = g^{pq} \nabla_p \nabla_q X$ with respect to the metric $g$ of~\eqref{eq:general-bundle-metric} is $\Delta X = $
\begin{equation*}
\Big[ \Big( \frac{1}{k^2} \Big) s'' + \Big( \frac{m-1}{k^2 r} + \frac{n h'}{h k^2} + \frac{mk'}{k^3} \Big) s' + \Big( - \frac{m-1}{k^2 r^2} + \frac{n h''}{h k^2} - \frac{n (h')^2}{h^2 k^2} + \frac{m k''}{k^3} - \frac{m(k')^2}{k^4} \Big) s \Big] \dib{r}.
\end{equation*}
\end{cor}
\begin{proof}
We use the same notation as in the proof of Corollary~\ref{cor:general-Laplacian}. For any vector field $X$, we have
\begin{equation*}
\Delta X = g^{ab} \nabla_a (\nabla_b X) - g^{ab} \Gamma^c_{ab} \nabla_c X.
\end{equation*}
If $X = F Y$ for some function $F$ and vector field $Y$, then from the above we get the well-known formula
$$ \Delta (F Y) = (\Delta F) Y + 2 g^{ab} (\nabla_a F) (\nabla_b Y) + F (\Delta Y). $$
Take $F = s k^2$ and $Y = k^{-2} \dbydr$, so that $X = F Y = s \dbydr$. Then we have
\begin{equation} \label{eq:general-Lap-1-temp}
\Delta X = \big( \Delta (s k^2) \big) \frac{1}{k^2} \dib{r} + 2 g^{ab} \nabla_a (s k^2) \nabla_b \Big( \frac{1}{k^2} \dib{r} \Big) + (sk^2) \Delta \Big( \frac{1}{k^2} \dib{r} \Big).
\end{equation}
We compute the three terms above one at a time.

Using~\eqref{eq:general-Laplacian} with $s$ replaced by $sk^2$ gives
\begin{align*}
\Delta (s k^2) & = \frac{1}{k^2} (sk^2)'' + \Big( \frac{m-1}{k^2 r} + \frac{n h'}{h k^2} + \frac{(m-2)k'}{k^3} \Big) (sk^2)' \\
& = \frac{1}{k^2} (s'' k^2 + 4 s' k k' + 2 s (k')^2 + 2 s k k'') + \Big( \frac{m-1}{k^2 r} + \frac{n h'}{h k^2} + \frac{(m-2)k'}{k^3} \Big) (s' k^2 + 2 s k k') \\
& = s'' + \Big( \frac{m-1}{r} + \frac{n h'}{h} + \frac{(m+2)k'}{k} \Big) s' + \Big( \frac{2(m-1)k'}{kr} + \frac{2n h' k'}{hk} + \frac{2(m-1) (k')^2}{k^2} + \frac{2 k''}{k} \Big) s,
\end{align*}
so we get
\begin{equation} \label{eq:general-Lap-1-temp2}
\begin{aligned}
\big( \Delta (s k^2) \big) \frac{1}{k^2} \dib{r} & = \frac{1}{k^2} s'' \dib{r} + \Big( \frac{m-1}{k^2 r} + \frac{n h'}{h k^2} + \frac{(m+2)k'}{k^3} \Big) s' \dib{r} \\
& \qquad {} + \Big( \frac{2(m-1)k'}{k^3r} + \frac{2n h' k'}{hk^3} + \frac{2(m-1) (k')^2}{k^4} + \frac{2 k''}{k^3} \Big) s \dib{r}.
\end{aligned}
\end{equation}
Since $F = s k^2$ is a function of $r$, we have $\nabla_i F = 0$ by~\eqref{eq:nabis}, and thus
\begin{align*}
2 g^{ab} \nabla_a (s k^2) \nabla_b \Big( \frac{1}{k^2} \dib{r} \Big) & = 2 g^{\alpha \beta} \nabla_{\alpha} (s k^2) \nabla_{\beta} \Big( \frac{1}{k^2} \dib{r} \Big) \\
& = 2 k^{-2} \nabla_{\alpha} (s k^2) \nabla_{\alpha} \Big( \frac{1}{k^2} \dib{r} \Big) \\
& = 2 k^{-2} (k^2 \nabla_{\alpha} s + 2 s k \nabla_{\alpha} k) \Big( - 2 k^{-3} (\nabla_{\alpha} k) \dib{r} + k^{-2} \nabla_{\alpha} \dib{r} \Big),
\end{align*}
which, using~\eqref{eq:nabalphas}, further simplifies to
\begin{align*}
2 g^{ab} \nabla_a (s k^2) \nabla_b \Big( \frac{1}{k^2} \dib{r} \Big) & = 2 k^{-2} \Big( \frac{k^2 x_{\alpha} s'}{r} + \frac{2 s k x_{\alpha} k'}{r} \Big) \Big( \frac{- 2 x_{\alpha} k'}{k^3 r} \dib{r} + \frac{1}{k^2} \nabla_{\alpha} \dib{r} \Big) \\
& = - \frac{4 k' s'}{k^3} \dib{r} - \frac{8 (k')^2 s}{k^4} \dib{r} + \Big( \frac{2 s'}{k^2} + \frac{4 k' s}{k^3} \Big) \frac{x_{\alpha}}{r} \nabla_{\alpha} \dib{r}.
\end{align*}
From the fact that $\dbydr = \frac{x_{\alpha}}{r} \sv_{\alpha}$ and Corollary~\ref{cor:nrr}, we have $\frac{x_{\alpha}}{r} \nabla_{\alpha} \dbydr = \frac{k'}{k} \dbydr$, so after collecting terms the above expression becomes
\begin{equation} \label{eq:general-Lap-1-temp3}
2 g^{ab} \nabla_a (s k^2) \nabla_b \Big( \frac{1}{k^2} \dib{r} \Big) = - \Big( \frac{2 k'}{k^3} \Big) s' \dib{r} - \Big( \frac{4 (k')^2}{k^4} \Big) s \dib{r}.
\end{equation}
Finally, from Lemma~\ref{lemma:Lap-RF}, we have
\begin{equation} \label{eq:general-Lap-1-temp4}
\begin{aligned}
& \qquad (sk^2) \Delta \Big( \frac{1}{k^2} \dib{r} \Big) \\
& = \Big( - \frac{2(m-1) k'}{k^3 r} - \frac{m-1}{k^2 r^2} + \frac{n h''}{h k^2} - \frac{n (h')^2}{h^2 k^2} - \frac{2 n h' k'}{h k^3} + \frac{(m-2) k''}{k^3} - \frac{3(m-2)(k')^2}{k^4} \Big) s \dib{r}.
\end{aligned}
\end{equation}
Substituting equations~\eqref{eq:general-Lap-1-temp2},~\eqref{eq:general-Lap-1-temp3}, and~\eqref{eq:general-Lap-1-temp4} into~\eqref{eq:general-Lap-1-temp} and collecting terms, after some cancellation we obtain the result.
\end{proof}

\begin{cor} \label{cor:general-normsq-nablaX}
Assume again that the metric~\eqref{eq:general-bundle-metric} is Ricci-flat. Let $X = s(r) \frac{\partial}{\partial r}$. Then the quantity $| \nabla X |^2$ with respect to the metric $g$ of~\eqref{eq:general-bundle-metric} is
\begin{align*}
| \nabla X |^2 & = - \frac{(m - 1) s^2 k''}{k} + (s')^2 + \frac{2 s s' k'}{k} + \frac{(2m - 1) s^2 (k')^2}{k^2} + \frac{(m-1) s^2 k'}{kr} \\
& \qquad {} + \frac{n s^2 h' k'}{h k} + \frac{(m-1)s^2}{r^2} - \frac{n s^2 h''}{h} + \frac{n s^2 (h')^2}{h^2}.
\end{align*}
\end{cor}
\begin{proof}
With respect to a local orthonormal frame we have
\begin{align} \nonumber
| \nabla X |^2 & = \nabla_i X_j \nabla_i X_j = \nabla_i (X_j \nabla_i X_j) - X_j \nabla_i \nabla_i X_j \\ \nonumber
& = \nabla_i \big( \tfrac{1}{2} \nabla_i (X_j X_j) \big) - X_j (\Delta X)_j \\ \label{eq:general-normsq-nablaX-1}
& = \tfrac{1}{2} \Delta (|X|^2) - \langle X, \Delta X \rangle.
\end{align}
For $X = s(r) \frac{\partial}{\partial r}$, we get from~\eqref{eq:normdbydr} that $|X|^2 = s^2 k^2$, and thus Corollary~\ref{cor:general-Laplacian} gives
\begin{align*}
\Delta (|X|^2) & = \Delta (s^2 k^2) \\
& = \frac{1}{k^2} (s^2 k^2)'' + \Big( \frac{m-1}{k^2 r} + \frac{n h'}{h k^2} + \frac{(m-2)k'}{k^3} \Big) (s^2 k^2)' \\
& = \frac{1}{k^2} (2 s s'' k^2 + 2 (s')^2 k^2 + 8 s s' k k' + 2 s^2 (k')^2 + 2 s^2 k k'') \\
& \qquad {} + \Big( \frac{m-1}{k^2 r} + \frac{n h'}{h k^2} + \frac{(m-2)k'}{k^3} \Big) (2 s s' k^2 + 2 s^2 k k'),
\end{align*}
which simplifies further to
\begin{align} \nonumber
\Delta (|X|^2) & = 2 s s'' + 2 (s')^2 + \frac{8 s s' k'}{k} + \frac{2 s^2 (k')^2}{k^2} + \frac{2 s^2 k''}{k} \\ \nonumber
& \qquad {} + \frac{2(m-1) s s'}{r} + \frac{2 n h' s s'}{h} + \frac{2(m-2) k' s s'}{k} \\ \label{eq:Delta-normsqX}
& \qquad {} + \frac{2(m-1) s^2 k'}{kr} + \frac{2 n h' s^2 k'}{h k} + \frac{2(m-2) s^2 (k')^2}{k^2}.
\end{align}
Using Corollary~\ref{cor:general-Laplacian-1} and~\eqref{eq:normdbydr} we also have
\begin{align} \nonumber
& \qquad \langle X, \Delta X \rangle = \Big \langle s \dib{r}, \Delta X \Big \rangle \\ \nonumber
& = s k^2 \Big[ \Big( \frac{1}{k^2} \Big) s'' + \Big( \frac{m-1}{k^2 r} + \frac{n h'}{h k^2} + \frac{mk'}{k^3} \Big) s' + \Big( - \frac{m-1}{k^2 r^2} + \frac{n h''}{h k^2} - \frac{n (h')^2}{h^2 k^2} + \frac{m k''}{k^3} - \frac{m(k')^2}{k^4} \Big) s \Big] \\ \nonumber
& = s s'' + \frac{(m-1) s s'}{r} + \frac{n h' s s'}{h} + \frac{m k' s s'}{k} \\ \label{eq:general-normsq-nablaX-2}
& \qquad {} - \frac{(m-1)s^2}{r^2} + \frac{n h'' s^2}{h} - \frac{n (h')^2 s^2}{h^2} + \frac{m s^2 k''}{k} - \frac{m s^2 (k')^2}{k^2}.
\end{align}
Substituting equations~\eqref{eq:Delta-normsqX} and~\eqref{eq:general-normsq-nablaX-2} into~\eqref{eq:general-normsq-nablaX-1} and collecting terms, after some cancellation we obtain the result.
\end{proof}

\section{Isometric solitons on Euclidean $\R^7$} \label{R7sec}

In this section we consider the case $M = \R^7$ such that the initial $\phir$ is the standard torsion-free $\G$-structure inducing the Euclidean metric $g$, and look for radially symmetric solutions of the isometric flow soliton equation~\eqref{solitoncond}.

\subsection{Derivation of the radially symmetric soliton equation on $\R^7$} \label{R7derivationsec}

We take the usual (global) Euclidean coordinates $x^1, \ldots, x^7$ on $\R^7$, in terms of which the unit radial vector field $\frac{\partial}{\partial r}$ is related to the Euler vector field $\vR$ by
$$\vR = x^i \dib{x^i} = r \dib{r}.$$

\begin{lemma} \label{lemma:R7Y}
Let $(\varphi, Y, c)$ be a soliton for the isometric flow on $\R^7$ with the Euclidean metric $g$. Then $Y = c \vR + Y_0$, where $Y_0$ is a Killing vector field on $(\R^7, g)$.
\end{lemma}
\begin{proof}
In terms of the coordinates $x^1, \ldots, x^7$, the equation $\mathcal L_Y g = 2 c g$ becomes $\frac{\partial}{\partial x^i} Y_j + \frac{\partial}{\partial x^j} Y_i = 2 c \, \delta_{ij}$. It is easy to verify that the solutions are $Y_i = c x^i + a_{ij} x^j + b_i$ where $a_{ij}$ is skew-symmetric. Thus $Y_0 = a_{ij} x^j \frac{\partial}{\partial x^i} + b_i \frac{\partial}{\partial x^i}$ generates a rigid motion of $(\R^7, g)$.
\end{proof}

The standard action of $\mathrm{SO}(7)$ on $\R^7$ acts by cohomogeneity one. The principal orbits are the spheres $\{ r = r_0 \}$ for $r_0 > 0$, where $r^2 = (x^1)^2 + \cdots + (x^7)^2$. We look for solutions to the isometric soliton system~\eqref{solitoncond} which are invariant under this action. More precisely, we want the data $(f, X, Y, c)$ to be $\mathrm{SO}(7)$-invariant. Since the space of $\mathrm{SO}(7)$-invariant vector fields is spanned by $\frac{\partial}{\partial r}$, we must have
\begin{equation} \label{radialXform}
X = a(r)\dib{r}
\end{equation}
for some smooth function $a$. Since $\vR = r \frac{\partial}{\partial r}$ is smooth at the origin, for $X = a(r) \frac{\partial}{\partial r}$ to be smooth at the origin, we need $a(r)/r$ to be a smooth even function of $r$, so $a$ must be a smooth odd function of $r$. Since $\frac{\partial}{\partial r}$ is a unit vector field, we have $|X|^2 = a^2$, and thus $f^2 + a^2 = 1$.

Similarly we must have $Y = b(r) \frac{\partial}{\partial r}$, but then Lemma~\ref{lemma:R7Y} forces $Y_0 = 0$ and $Y = c \vR$, because there are no non-zero radially symmetric Killing fields in Euclidean space. Since $Y = c \vR$, we have $Y = c x^i \frac{\partial}{\partial x^i}$, so $\nabla_i Y_j = c \delta_{ij}$ is symmetric, and $(\curl_{\phii} Y)_m = \nabla_i Y_j \phii_{ijm} = 0$.

Hence we have satisfied the first equation in~\eqref{solitoncond} with $Y = c \vR$, and since $\curl_{\phii} Y = 0$ we have reduced the second equation in~\eqref{solitoncond} to
\begin{equation} \label{R7Tcond}
\tdiv T =c\, r \dib{r} \hk T.
\end{equation}
As we will see, using the expressions~\eqref{simpleT} and~\eqref{simpledivT} for the torsion and its divergence, the equation~\eqref{R7Tcond} simplifies to a second-order ODE for $a$ and $f$, subject to the constraint that $a^2+f^2=1$.

\begin{prop} \label{R7ODE}
The soliton equation~\eqref{R7Tcond} is equivalent to the second order nonlinear ODE
\begin{equation} \label{hradialOdeFirst}
\dfrac1{r}(a f'' - a'' f) +\dfrac{6}{r^2}(a f' - a' f)+ \dfrac{6}{r^3} a f = -c \dfrac{a'}{f}.
\end{equation}
\end{prop}
Before proceeding with the proof, we make some observations. Note that the condition that $a$ be a smooth odd function of $r$, and $f^2=1-a^2$, imply that each
side of~\eqref{hradialOdeFirst} is smooth near $r=0$. To emphasize that this is actually an equation for one unknown function, we note that we can eliminate the derivatives of $f$ by differentiating $f^2 + a^2=1$ to get
\begin{equation} \label{fprime}
f' = -\dfrac{a a'}{f}.
\end{equation}
It follows that 
\begin{equation} \label{firstfhrule}
a f' - a' f =-\dfrac{(a^2 + f^2) a'}f = -\dfrac{a'}f
\end{equation}
and by differentiating again we get
\begin{equation} \label{secondfhrule}
a f'' - a'' f = - \dfrac{a''}f + \dfrac{a' f'}{f^2} = - \dfrac{a''}f - \dfrac{a (a')^2}{f^3}.
\end{equation}
These allow us to rewrite~\eqref{hradialOdeFirst}, after cancelling factors of $-1$ and multiplying by $fr$, as 
\begin{equation} \label{hradialOdeSecond}
a'' + \dfrac{a}{f^2} (a')^2 + (6-c r^2) \dfrac{a'}{r} = 6 \dfrac{f^2}{r^2} a.
\end{equation}

\begin{proof}[Proof of Proposition~\ref{R7ODE}]
The Euclidean metric on $\R^7 \setminus \{ 0 \}$ can be expressed as a warped product
$$g = r^2 \gN + (dr)^2,$$
where $\gN$ is the constant-curvature one metric on $S^6$. Thus, $g$ has the form~\eqref{eq:general-bundle-metric} of the bundle metrics discussed in Section~\ref{Laplacian-sec} in the special case where
$$ n=6, \qquad m=1, \qquad h(r) = r, \qquad k(r) = 1. $$
We use the computations of Section~\ref{Laplacian-sec} in what follows.

Because~\eqref{eq:radialXderiv} shows that $\nr X$ is a multiple of $X$, by Corollary~\ref{cor:T-phidropout} we are left with
\begin{equation} \label{eq:radialhookTnophi}
\dib{r} \hk T = 2 (\nr f) X^{\flat} - 2 f \nr X^{\flat}.
\end{equation}
Using~\eqref{eq:radialXderiv} with $k=1$ and~\eqref{firstfhrule} we get
\begin{equation} \label{eq:radialhookTflat}
\dib{r} \hk T = 2 (f' - f \dfrac{a'}{a}) X^{\flat} = 2(a f' - a' f ) dr = -2\dfrac{a'}{f} dr.
\end{equation}
Next, specializing Corollaries~\ref{cor:general-Laplacian} and~\ref{cor:general-Laplacian-1} to this case gives
\begin{equation} \label{flatlaplacians}
\Delta f = f'' + \dfrac{6}r f', \qquad \Delta X = \left( a'' + \dfrac6{r} a' - \dfrac6{r^2} a \right) \dib{r}.
\end{equation}
Since $\Delta X$ is a multiple of $X$, by Corollary~\ref{cor:T-phidropout} we have
\begin{align}
\tdiv T & = 2(\Delta f) X^{\flat} - 2 f \Delta X^{\flat} \notag \\
& = 2 \left( \big( f'' + \dfrac{6}{r} f' \Big) a - f \Big( a'' + \dfrac{6}{r} a' - \dfrac{6}{r^2} a \Big) \right) dr \notag \\
& = 2 \left( (a f'' - a'' f) + \dfrac{6}{r} ( a f' - a'f) + \dfrac{6}{r^2} a f \right) dr. \label{simplerdivT}
\end{align}
Substituting this expression and~\eqref{eq:radialhookTflat} into~\eqref{R7Tcond}, and dividing out by $2r$, yields the ODE~\eqref{hradialOdeFirst}.
\end{proof}

For later use, we compute the square norm of the torsion tensor. In this special case Corollary~\ref{cor:general-normsq-nablaX} gives
$$ |\nabla X|^2 = (a')^2 + 6 \dfrac{a^2}{r^2}. $$
Hence from~\eqref{normT} and~\eqref{fprime} we have
\begin{equation} \label{normTradial}
\tfrac{1}{4} |T|^2 = |\nabla X|^2 +|\nabla f|^2 = (a')^2 +6 \dfrac{a^2}{r^2} + (f')^2 = \dfrac{(a')^2}{f^2} + 6 \dfrac{a^2}{r^2}.
\end{equation}

\subsection{Analysis of the radially symmetric soliton equation on $\R^7$} \label{analysisR7sec}

Since $f^2 + a^2 = 1$, we can simplify the ODE~\eqref{hradialOdeSecond} so that it involves only a single unknown function by making the substitutions
\begin{equation} \label{trigsub}
a(r) = \sin\big( \tfrac{1}{2} u(r) \big) \quad \text{and} \quad f(r) = \cos\big( \tfrac{1}{2} u(r) \big),
\end{equation}
whereupon the ODE becomes 
\begin{equation} \label{firstradialODE}
r^{2} u'' + (6 - c r^{2}) r u' -6 \sin u=0.
\end{equation}

Recall that for $X$ to be smooth at the origin, we require that $a(r)$ be a smooth odd function of $r$, which is equivalent to $u(r)$ being a smooth odd function of $r$. In Section~\ref{ODEsec} we show that~\eqref{firstradialODE} has $r=0$ as a regular singular point, and we prove that, for any value of $c$,~\eqref{firstradialODE} admits a one-parameter family of analytic solutions $u(r)$ which are odd functions of $r$ defined for all real values of $r$. (The free parameter is $u'(0)$. If this is zero, then the solution $u(r)$ is identically zero, and the soliton is trivial.) It follows that there exist global radially symmetric expanding, steady, and shrinking solitons.

We have thus established the following:
\begin{thm} \label{thm:R7final}
For any $c \in \R$, there exists a one-parameter family of global radially symmetric solitons for the isometric flow on the Euclidean $\R^7$. That is, there is a unique global solution to the ODE~\eqref{hradialOdeSecond} satisfying the initial conditions $a(0) = 0$ and $a'(0) \neq 0$.
\end{thm}
Making the substitution~\eqref{trigsub} in the expression~\eqref{normTradial} for $\frac{1}{4} |T|^2$
gives
\begin{equation} \label{flat:squarenormT}
|T|^2 = \left( \dfrac{du}{dr} \right)^2 + 24 \dfrac{\sin^2( \tfrac{1}{2} u)}{r^2}.
\end{equation}
In Section~\ref{ODEsec}, we also prove the following:

\begin{prop} \label{prop:asymptotic-torsion}
Consider the radially symmetric solitons for the isometric flow on the Euclidean $\R^7$ given by Theorem~\ref{thm:R7final}, with torsion $T$. If $c \leq 0$ then $\lim_{r \to \infty} |T|^2 = 0$. If $c > 0$ then $\lim_{r \to \infty} |T|^2 = \infty$.
\end{prop}

\subsection{The isometric flow on $\R^7$ in the radially symmetric case} \label{R7flowsec}

Our expression~\eqref{simplerdivT} for the vector field $\tdiv T$ in the radially symmetric case gives
$$\tdiv T = 2 \left[ - \dfrac{a''}{f } - \dfrac{a (a')^2}{f^3 } - 6 \dfrac{a'}{f r} + 6 \dfrac{fa}{r^2}\right] \dib{r}$$
after eliminating derivatives of $f$ using~\eqref{firstfhrule} and~\eqref{secondfhrule}. Again, since $X =a \dbydr$ and $\tdiv T$ are everywhere linearly dependent, the cross product term in the formula~\eqref{isoflowX} for the flow of $X$ vanishes, and thus the isometric flow for radially symmetric data on the Euclidean $\R^7$ specializes to
$$\dot X = -\tfrac{1}{2} f \tdiv T = -\left[ - a'' - \dfrac{a (a')^2}{f^2} - 6 \dfrac{a'}{r} +6\dfrac{f^2 a}{r^2}\right] \dib{r}.$$
Substituting $X =a \dbydr$, we get
\begin{equation} \label{radialhdot}
\dot a = a'' + \dfrac{a (a')^2}{f^2} + 6 \dfrac{a'}r - 6 \dfrac{f^2 a}{r^2}.
\end{equation}
Furthermore, under our trigonometric substitution $a = \sin(\frac{1}{2} u)$ this is equivalent to
\begin{equation} \label{radialydot}
\dot u = u''+6 \dfrac{u'}{r} -6 \dfrac{\sin u}{r^2}.
\end{equation}
We suspect that there are conditions under which this flow converges to a radially symmetric soliton. Proving this may be facilitated by examining the behavior under the flow of the quantity
\begin{equation} \label{radialQdef}
Q = \dfrac{a''}{r a'} +\dfrac{a a'}{f^2 r} + \dfrac{6}{r^2} - \dfrac{6f^2 a}{r^3 a'}
\end{equation}
which by~\eqref{hradialOdeSecond} equals $c$ along the solitons and is thus \emph{constant} (and vanishes on the steady soliton). 

\section{Isometric solitons from $\mathrm{SU}(3)$-structures} \label{SU3sec}

In this section we consider two cases where $M^7 = L^1 \times N^6$ is equipped with a torsion-free $\G$-structure which is induced from a particular $\mathrm{SU}(3)$-structure on a $6$-manifold $N$ as a warped product:
\begin{enumerate}[{$[$}a{$]$}]
\item When $L^1 = \R$ or $L^1 = S^1$ and $N^6$ is equipped with a Calabi--Yau structure. There is a torsion-free $\G$-structure on $M^7 = L^1 \times N^6$ inducing the product metric.
\item When $L^1 = (0, \infty)$ and $N^6$ is equipped with a nearly K\"ahler structure. There is a torsion-free $\G$-structure on $M^7 = (0, \infty) \times N^6$ inducing the cone metric.
\end{enumerate}
In both cases, we look for solutions of the isometric flow soliton equation~\eqref{solitoncond} for which the data depends only on the coordinate $r$ on $L^1$. 

We first collect some general facts about warped product $\G$-structures over $L^1 \times N^6$ where $N^6$ is equipped with an $\mathrm{SU}(3)$-structure. We use the same notation as in Karigiannis--McKay--Tsui~\cite{KMT}.

In what follows, the underlying manifold $M$ will be a product $L^1 \times N^6$, where $L$ is a circle or a line and $N$ is endowed with an $\mathrm{SU}(3)$-structure. That is, $N$ carries a Riemannian metric $\gN$, a $\gN$-compatible almost complex structure $\rJ$, a K\"ahler form $\omega$ satisfying the usual compatibility condition $\omega(X,Y) = \gN(\rJ X, Y)$, and a nowhere vanishing $(3,0)$-form $\Omega$ such that
$${\mathrm{vol}}_{\gN} = \tfrac16 \omega^3 = \tfrac18 \ri \Omega \wedge \overline{\Omega}.$$

We define a class of $\G$-structures on $M$ as follows. Denoting by $r$ a coordinate along $L$, let $F(r)$ be a smooth nowhere-vanishing complex-valued function, and let $G(r)$ be a smooth real positive function. Then the 3-form
\begin{equation} \label{eq:SU3-varphi}
\varphi=\Re(F^3 \Omega) -G |F|^2 dr \wedge \omega
\end{equation}
induces on $M$ the Riemannian metric
$$g= G^2\, dr^2 + |F|^2 \gN$$
which is a warped product over the base $L$. The corresponding 4-form is
$$\psi = -G \,dr \wedge \Im( F^3 \Omega) - \tfrac{1}{2} |F|^4 \omega^2.$$

In the following subsections we construct solitons and describe the isometric flow for two sub-classes of such $\G$-structures, which are both torsion-free.

\subsection{Calabi--Yau fibrations} \label{CYsec}

In this case we assume that $(N, \gN, \rJ, \omega, \Omega)$ is a \emph{Calabi--Yau} $3$-fold. Thus $\rJ$ is integrable and $d\omega=0$, so that $\gN$ is a K\"ahler metric, and moreover $d\Omega = 0$ and $\gN$ is \emph{Ricci-flat}.

As before, we consider $\G$-structures isometric to a torsion-free background. In this case, the background is the product metric
$$g = dr^2 + \gN$$
corresponding to $F=G=1$. Thus, this is a special case of the bundle metric~\eqref{eq:general-bundle-metric} where
$$ n=6, \qquad m=1, \qquad h(r) = 1, \qquad k(r) = 1. $$

\textbf{Isometric solitons.} We suppose that our isometric soliton is described by data depending only on the coordinate $r$ on $L^1$. That is, we impose that
\begin{equation} \label{CY-symm}
X = a(r) \dib{r} \quad \text{and} \quad Y = b(r) \dib{r}.
\end{equation}
Hence, we have $|X|^2 = a^2$ and $f^2 = 1 - |X|^2 = 1 - a^2$. 
\begin{prop} Let $\phir$ be the $\G$-structure given by~\eqref{eq:SU3-varphi} for $F=G=1$, let $\phii$ be the isometric structure given by~\eqref{isog2form}, and
let $(\phii,Y,c)$ satisfy~\eqref{solitoncond} where $X$ and $Y$ are given by~\eqref{CY-symm}. Then $c=0$, so there can only be \emph{steady} solitons of this form. Moreover, $b$ is a constant, and $a(r)$ is a solution of the ODE
\begin{equation} \label{CY-simpleODE}
a'' - b a'+ \dfrac{a (a')^2}{1-a^2}=0.
\end{equation}
Finally, the torsion of $\phii$ satisfies
\begin{equation} \label{CY-Tnorm}
\tfrac{1}{4} |T|^2= \dfrac{(a')^2}{1-a^2}.
\end{equation}
\end{prop}
\begin{proof}
Using Lemma~\ref{lemma:LYg} and the fact that $m=1$ and $h = k = 1$, the first soliton condition in~\eqref{solitoncond} becomes $c = 0$ and $b' = 0$, so $b$ is a \emph{constant}. Since $\curl Y = \left[ \hstar (d Y^{\flat}) \wedge \psi\right]^{\sharp}$ from~\eqref{curldefn1}, and in this case $Y^{\flat} = b \,dr$ is closed, we see that $\curl Y=0$. Thus, the second soliton condition in~\eqref{solitoncond} simplifies to
\begin{equation} \label{simplerCYcond}
\tdiv T = b\dib{r} \hk T.
\end{equation}
In this case,~\eqref{eq:radialXderiv} shows that $\nr X = a' \dbydr$. Since this is a multiple of $X$, applying Corollary~\ref{cor:T-phidropout} shows that~\eqref{eq:radialhookTnophi} holds. Specializing to this case and using~\eqref{firstfhrule} yields
\begin{equation} \label{CY-T}
\dib{r} \hk T = 2(a f' - a'f ) dr = - 2 \frac{a'}{f} dr.
\end{equation}
Specializing Corollaries~\ref{cor:general-Laplacian} and~\ref{cor:general-Laplacian-1} to this case gives
$$ \Delta f = f'', \qquad \Delta X = a'' \dib{r}. $$
Since $\Delta X$ is a multiple of $X$, using Corollary~\ref{cor:T-phidropout} and~\eqref{secondfhrule} yields
\begin{equation} \label{CY-divT}
\tdiv T = 2(\Delta f) X^{\flat} - 2 f \Delta X^{\flat} =(2a f'' - 2a'' f) dr = \left( -2\dfrac{a''}f-2\dfrac{a (a')^2}{f^3}\right) dr.
\end{equation}
The soliton condition~\eqref{simplerCYcond} then becomes
$$-2\dfrac{a''}f-2\dfrac{a (a')^2}{f^3} = -2 b \dfrac{a'}{f}.$$
Using $f^2=1-a^2$ leads to the ODE~\eqref{CY-simpleODE}.

Specializing Corollary~\ref{cor:general-normsq-nablaX} to this case gives $|\nabla X|^2 = (a')^2$, and then from~\eqref{normT} we have
$$\tfrac{1}{4} |T|^2 = |\nabla X|^2 + |\nabla f|^2 = (a')^2 + (f')^2.$$
Using the relation $a^2+f^2=1$ to eliminate $f$ yields~\eqref{CY-Tnorm}.
\end{proof}

Note that if we set $a(r) = \sin(u(r))$, then the ODE~\eqref{CY-simpleODE} is transformed to 
\begin{equation} \label{CY-soliton-eq}
u'' - b u'=0
\end{equation}
with solution $u(r)=c_0 +c_1 \re^{br}.$

In the case where $L = S^1$, the only periodic solutions are those where $a$, and hence also $f$, are constant. By~\eqref{CY-Tnorm}, these solitons are torsion-free. They are product $\G$-structures on $S^1 \times N^6$ which correspond to phase rotations of the holomorphic volume form $\Omega$ on $N^6$ (see~\cite[Section 3.3]{K-thesis}).

In the case where $L = \R$ and $b,c_1\ne 0$ we have non-trivial solitons. By reversing the $r$ coordinate if necessary, we can assume that 
$b>0$. Then from~\eqref{CY-Tnorm} we obtain that
$$ \tfrac{1}{4} |T|^2 = c_1^2 b^2 \re^{2br}, $$
so $|T|$ vanishes as $r \to -\infty$, whereas it is unbounded as $r \to+\infty$.

We have thus established the following.

\begin{thm} \label{thm:CYfinal}
Let $M^7 = \R \times N^6$, where $N$ is a Calabi--Yau $3$-fold, and $M$ is equipped with the standard torsion-free product $\G$-structure $\varphi = \Re(\Omega) - dr \wedge \omega$. The only isometric solitons in the isometry class of $[\varphi]$ of the special form~\eqref{CY-symm} are \emph{steady}, and are given by
$$ a(r) = \sin (c_0 + c_1 \re^{br}) \quad \text{and} \quad b(r) = b $$
for constants $b, c_0, c_1$. Moreover, the pointwise norm $|T|$ of the torsion of this solution goes to zero at one end of $\R$ and becomes unbounded at the other end.
\end{thm}

\begin{rmk} \label{rmk:CY-more-general}
More generally, we could suppose that the vector field $Y$ in the soliton conditions~\eqref{solitoncond} takes the form
$$ Y = b(r) \dib{r} + \lambda(r) Z $$
where $Z$ is a vector field on $N$. Expanding the condition $\Lie_Y g =2c g$ gives
$$ 2 b' (dr)^2 + \lambda \Lie_Z \gN = 2 c (dr)^2 + 2 c \gN, $$
implying that $b' = c$ and $\Lie_Z \gN = \frac{2 c}{\lambda} \gN$. Hence, $b(r) = c r + b_0$, $\lambda$ is a nonzero constant, and $Z$ is a conformal Killing field on $N$. However, we claim that if $N$ is compact then necessarily $c=0$. To see this, we trace the conformal Killing condition
$$ \nablaN_i Z_j + \nablaN_j Z_i = \dfrac{2c}{\lambda} \gN_{ij} $$
to obtain $d^* Z^{\flat} = -6c/\lambda$. Then integrating over $N$ gives $c=0$. Hence, $Z$ is a Killing field and $N$ has reduced holonomy.
\end{rmk}

\textbf{The isometric flow.} Consider the isometric flow among $\G$-structures isometric to the product metric, obtained from $\varphi$ by~\eqref{isog2form} using a vector field $X$ of the special form $X = a \frac{\partial}{\partial r}$. Our computations above show that $X$ and $\tdiv T$ are everywhere linearly dependent, so the cross product term in~\eqref{isoflowX} for the flow of $X$ vanishes, and thus by~\eqref{CY-divT} the isometric flow in this case specializes to
$$\dot{a} = a'' +\dfrac{a(a')^2}{1-a^2}.$$
Remarkably, if we again make the substitution $a(r)=\sin u(r)$, this becomes the heat equation
$$ \dot{u} = u''. $$

In the case $L = S^1$, we can use Fourier series to show that $u$ converges exponentially to a constant. In the case $L = \R$, we know that if $u$ is asymptotic to the same constant at both ends then it will converge exponentially to that constant value.

\begin{rmk}
If the initial data is strictly monotone, it may be possible to show that the isometric flow converges to the soliton describe above. By~\eqref{CY-soliton-eq}, on this soliton the quantity $Q = \tfrac{u''}{u'}$ is constant. Under the heat equation $\dot{u} = u''$, a computation shows that $\dot{Q} = Q'' +2 Q Q'$. Under this evolution, one would need to prove that $Q$ converges to a constant as $t \to +\infty$.
\end{rmk}

\subsection{Nearly K\"ahler fibrations} \label{NKsec}

In this case we assume that $(N, \gN, \rJ, \omega, \Omega)$ is a \emph{nearly K\"ahler} $6$-manifold. This means that the $\mathrm{SU}(3)$-structure satisfies
$$d\omega= -3\Re(\Omega), \qquad d\Im(\Omega) = 2\omega^2.$$

As before, we consider $\G$-structures isometric to a torsion-free background. In this case, the background is the \emph{cone} metric
$$g = dr^2 + r^2 \gN $$
corresponding to $F(r) = r$ and $G=1$. Thus, this is a special case of the bundle metric~\eqref{eq:general-bundle-metric} where
$$ n=6, \qquad m=1, \qquad h(r) = r, \qquad k(r) = 1. $$

\textbf{Isometric solitons.} We again suppose that our isometric soliton is described by data depending only on the coordinate $r$ on $L^1$. That is, we impose that
\begin{equation} \label{NK-symm}
X = a(r) \dib{r} \quad \text{and} \quad Y = b(r) \dib{r}.
\end{equation}
Hence, by~\eqref{eq:normdbydr} we have $|X|^2 = a^2$ and $f^2 = 1 - |X|^2 = 1 - a^2$. 

\begin{prop} \label{prop:NK-solitons}
Let $\phir$ be the $\G$-structure given by~\eqref{eq:SU3-varphi} for $F(r)=r$ and $G=1$, let $\phii$ be the isometric structure given by~\eqref{isog2form}, and
let $(\phii,Y,c)$ satisfy~\eqref{solitoncond} where $X$ and $Y$ are given by~\eqref{NK-symm}. Then $b(r) = cr$, and $a(r)$ is a solution of the ODE
\begin{equation} \label{NKode}
a'' + \dfrac{a}{f^2} (a')^2 + (6-c r^2) \dfrac{a'}{r} = 6 \dfrac{f^2}{r^2} a,
\end{equation}
which is the same as~\eqref{hradialOdeSecond}. Moreover the torsion satisfies
$$ \tfrac{1}{4}|T|^2 = \dfrac{(a')^2}{f^2} + 6 \dfrac{a^2}{r^2}, $$
which is the same as~\eqref{normTradial}.
\end{prop}
\begin{proof}
Using Lemma~\ref{lemma:LYg} and the fact that $m=1$, $h = r$, and $k = 1$, the first soliton condition in~\eqref{solitoncond} becomes $c r = b$ and $c = b'$, and thus $b = c r$ is the only constraint. As in the Calabi-Yau case, since $Y^{\flat} = b(r) \,dr$ then $\curl Y=0$. Thus, the second soliton condition in~\eqref{solitoncond} simplifies to
\begin{equation} \label{simplerNKcond}
\tdiv T = c r \dib{r} \hk T,
\end{equation}
which is the same as the condition~\eqref{R7Tcond} for radially symmetric solitons on Euclidean $\R^7$.

As in the Calabi-Yau case, $\nr X$ is a multiple of $X$,~\eqref{eq:radialhookTnophi} holds, and we compute that $\dbydr \hk T = -2\dfrac{a'}{f} dr$.
In this case, Corollaries~\ref{cor:general-Laplacian} and~\ref{cor:general-Laplacian-1} specialize to give
$$ \Delta f = f'' + \dfrac{6}r f', \qquad \Delta X = \left( a'' + \dfrac6{r} a' - \dfrac6{r^2} a \right) \dib{r}, $$
which coincide with the formula~\eqref{flatlaplacians} for the Euclidean case. Again applying Corollary~\ref{cor:T-phidropout} lets us obtain the same formula~\eqref{simplerdivT} for $\tdiv T$ as in the Euclidean case. Substituting our specializations into~\eqref{simplerNKcond} results in the same ODE as the Euclidean case.
\end{proof}

\begin{rmk} \label{rmk:NK}
In hindsight, this was to be expected, because the round $S^6$ is a nearly K\"ahler 6-manifold whose metric cone gives the Euclidean $\R^7$, so situation in Section~\ref{R7sec} is just a special case of this one.
\end{rmk}

Hence, by the arguments in Section~\ref{ODEsec}, we conclude that for any $c$ there exist smooth solutions to~\eqref{NKode} which are odd functions of $r$ defined for all $r\in\R$. Thus, the vector fields~\eqref{NK-symm} are smooth on $M$, as is the $\G$-structure $\phii$. Since the norm of the torsion is the same as in the Euclidean case, we conclude that if $c \leq 0$ then $\lim_{r \to \infty} |T|^2 = 0$, while if $c > 0$ then $\lim_{r \to \infty} |T|^2 = \infty$.

\textbf{The isometric flow.} Consider the isometric flow among $\G$-structures isometric to the cone metric, obtained from $\varphi$ by~\eqref{isog2form} using a vector field $X$ of the special form $X=a(r)\dib{r}$. Our computations above show that $X$ and $\tdiv T$ are everywhere linearly dependent, so the cross product term in~\eqref{isoflowX} for the flow of $X$ vanishes, and thus by~\eqref{simplerdivT} the isometric flow in this case specializes to
$$\dot{a} =a'' + \dfrac{a (a')^2}{1-a^2} +6\dfrac{a'}r - 6 \dfrac{a(1-a^2)}{r^2},$$
which is the same as~\eqref{radialhdot}. Again, the quantity $Q$ defined by~\eqref{radialQdef} is constant along the solitons.

\section{Isometric solitons on the Bryant--Salamon manifolds} \label{BSsec}

In this section we consider the case where $M^7$ is one of the Bryant--Salamon torsion-free $\G$-manifolds which were first discovered in~\cite{BS}. We follow the notation of Karigiannis--Lotay~\cite[Section 3]{KL}, except that we use $\lambda > 0$ instead of $c > 0$ to denote the parameter encoding the ``distance'' to the limiting asymptotic $\G$-cone.

There are two examples:

Case [a]: $M^7 = \Lambda^2_- (B^4)$, where $B^4$ is either the round $S^4$ or the Fubini--Study $\C\PR^2$
with their standard orientations.

Case [b]: $M^7 = \slashed{S} (S^3)$, the spinor bundle over the round $S^3$ with its standard orientation.

For any $\lambda > 0$, there is a torsion-free $\G$-structure $\varphi_{\lambda}$ on $M$, inducing a complete holonomy $\G$ metric $g_{\lambda}$ on $M$. For $\lambda = 0$, we get a torsion-free $\G$-structure $\varphi_0$ on $M' = M \setminus \{ \text{zero section} \}$, inducing an incomplete holonomy $\G$ metric $g_0$ on $M'$. In this case $(M', g_0)$ is a \emph{metric cone}, and for $\lambda > 0$ the Riemannian manifolds $(M, g_{\lambda})$ are asymptotically conical, with asymptotic cone $(M', g_0)$. To describe these $\G$-structures and metrics explicitly, we use the notation of Section~\ref{Laplacian-sec}.

In Case [a], let $\{ b_0, b_1, b_2, b_3 \}$ denote a local oriented orthonormal coframe for $B^4$, and let
$$\sigma_1 = b_0 \wedge b_1 - b_2 \wedge b_3, \quad \sigma_2 = b_0 \wedge b_2 - b_3 \wedge b_1, \quad
\sigma_3 = b_0 \wedge b_3 - b_1 \wedge b_2,$$
yielding a local frame of sections for $\Lambda^2_- (B^4)$. As in Section~\ref{Laplacian-sec} we let $\zeta_1, \zeta_2, \zeta_3$ be the corresponding vertical 1-forms on $M = \Lambda^2_- (B^4)$. Combined with the pullbacks of the $b_i$, these give a local coframe on $M$. There are smooth positive functions $h, k$ of $r \in [0, \infty)$ such that 
$\{ k \zeta_1, k \zeta_2, k \zeta_3, h b_0, h b_1, h b_2, h b_3 \}$ is a local $\G$-adapted orthonormal coframe for 
$\varphi_{\lambda}$. This means (omitting the $\wedge$ symbol for brevity) that
\begin{align} \label{eq:ph-BS-a}
\varphi_{\lambda} & = k^3 \zeta_1 \zeta_2 \zeta_3 + k h^2 \zeta_1 (b_0 b_1 - b_2 b_3) + k h^2 \zeta_1 (b_0 b_2 - b_3 b_1) + k h^2 \zeta_3 (b_0 b_3 - b_1 b_2), \\ 
\nonumber
g_{\lambda} & = k^2 (\zeta_1^2 + \zeta_2^2 + \zeta_3^2) + h^2 (b_0^2 + b_1^2 + b_2^2 + b_3^2).
\end{align}

In Case [b], let $\{ b_1, b_2, b_3 \}$ denote a left-invariant coframe on $S^3 \cong \mathrm{SU}(2)$, and thus also an oriented orthonormal coframe for the round metric. Denote again by $\{ b_1, b_2, b_3 \}$ their pullbacks to $\slashed{S} (S^3)$. As discussed in~\cite[Section 3]{KL}, this frame determines $\{ \zeta_0, \zeta_1, \zeta_2, \zeta_3 \}$, which is an orthonormal frame for the space of horizontal $1$-forms on $M = \slashed{S} (S^3)$ with respect to the fibre metric and connection on $M$ induced from the Riemannian metric on $S^3$. There are smooth positive functions $h, k$ of $r \in [0, \infty)$ such that $\{ h b_1, h b_2, h b_3, k \zeta_0, k \zeta_1, k \zeta_2, k \zeta_3 \}$ is a local $\G$-adapted orthonormal coframe for $\varphi_{\lambda}$. This means that
\begin{align} \label{eq:ph-BS-b}
\varphi_{\lambda} & = h^3 b_1 b_2 b_3 + h k^2 b_1 (\zeta_0 \zeta_1 - \zeta_2 \zeta_3) + h k^2 b_1 (\zeta_0 \zeta_2 - \zeta_3 \zeta_1) + h k^2 b_3 (\zeta_0 \zeta_3 - \zeta_1 \zeta_2), \\ \nonumber
g_{\lambda} & = h^2 (b_1^2 + b_2^2 + b_3^2) + k^2 (\zeta_0^2 + \zeta_1^2 + \zeta_2^2 + \zeta_3^2).
\end{align}
Note that the data $(\varphi_{\lambda}, \psi_{\lambda}, g_{\lambda})$ in case [b] is obtained from the pair in case [a] by swapping the roles of $h b$ and $k \zeta$. However, the precise nature of the functions $h$, $k$ in the torsion-free Bryant--Salamon solution is different in the two cases. Explicitly, both are special cases of the bundle metric~\eqref{eq:general-bundle-metric} where
\begin{equation} \label{eq:hk-functions}
\begin{aligned}
\text{Case [a]:} \quad n & = 4, & m & = 3, & h(r) & = \sqrt{2}(\lambda + r^2)^{\frac{1}{4}}, & k(r) & = (\lambda + r^2)^{-\frac{1}{4}}, \\
\text{Case [b]:} \quad n & = 3, & m & = 4, & h(r) & = \sqrt{3}(\lambda + r^2)^{\frac{1}{3}}, & k(r) & = 2 (\lambda + r^2)^{-\frac{1}{6}}.
\end{aligned}
\end{equation}

\textbf{Isometric solitons.} As usual we suppose that our isometric soliton is described by data depending only on the radial coordinate $r$. That is, we impose that
\begin{equation} \label{BS-symm}
X = a(r) \dib{r} \quad \text{and} \quad Y = b(r) \dib{r}.
\end{equation}
Hence, by~\eqref{eq:normdbydr} we have $|X|^2 = k^2 a^2$ and $f^2 = 1 - |X|^2 = 1 - k^2 a^2$. Because of this, in Proposition~\ref{prop:BS-solitons} below, when we make a trigonometric substitution in the differential equations that $a$ and $f$ must satisfy analogously to~\eqref{trigsub}, the appropriate substitution will instead be
\begin{equation}\label{BS-trigsub}
a = \dfrac1{k} \sin(\tfrac12 u)\qquad f = \cos(\tfrac12 u).
\end{equation}

\begin{prop} \label{prop:BS-LYg}
Consider the first soliton condition $\mathcal{L}_Y g = 2 c g$ in~\eqref{solitoncond} for the Bryant--Salamon torsion-free $\G$-manifolds, where $Y$ is of the form~\eqref{BS-symm}.
\begin{itemize}
\item If $\lambda > 0$ in either case, then we must have $c = 0$ and $b = 0$.
\item If $\lambda = 0$ in Case [a] then we must have $b = 2 c r$.
\item If $\lambda = 0$ in Case [b] then we must have $b = \frac{3}{2} c r$.
\end{itemize}
\end{prop}
\begin{proof}
Since $m > 1$, from Lemma~\ref{lemma:LYg} we deduce that the first soliton condition in~\eqref{solitoncond} yields $b = \mu r$ for some constant $\mu$, and moreover that
$$ c h = \mu r h' \qquad \text{and} \qquad c k = \mu r k' + \mu k. $$
In case [a], substituting $h, k$ from~\eqref{eq:hk-functions} into the above and simplifying yields
$$ c (\lambda + r^2)^{\frac{1}{4}} = \frac{\mu}{2} r^2 (\lambda + r^2)^{-\frac{3}{4}}, \qquad c (\lambda + r^2)^{- \frac{1}{4}} = - \frac{\mu}{2} r^2 (\lambda + r^2)^{- \frac{5}{4}} + \mu (\lambda + r^2)^{- \frac{1}{4}}. $$
These in turn become
$$ c (\lambda + r^2) = \frac{\mu}{2} r^2, \qquad c (\lambda + r^2) = - \frac{\mu}{2} r^2 + \mu (\lambda + r^2). $$
We see immediately that if $\lambda > 0$, then $c = \mu = 0$. If $\lambda = 0$, both equations give $\mu = 2 c$.

In case [b], substituting $h, k$ from~\eqref{eq:hk-functions} into the above and simplifying yields
$$ c (\lambda + r^2)^{\frac{1}{3}} = \frac{2 \mu}{3} r^2 (\lambda + r^2)^{-\frac{2}{3}}, \qquad c (\lambda + r^2)^{- \frac{1}{6}} = - \frac{\mu}{3} r^2 (\lambda + r^2)^{- \frac{7}{6}} + \mu (\lambda + r^2)^{- \frac{1}{6}}. $$
These in turn become
$$ c (\lambda + r^2) = \frac{2 \mu}{3} r^2, \qquad c (\lambda + r^2) = - \frac{\mu}{3} r^2 + \mu (\lambda + r^2). $$
Again we see that if $\lambda > 0$, then $c = \mu = 0$. If $\lambda = 0$, both equations give $\mu = \frac{3}{2} c$.
\end{proof}

\begin{rmk} \label{rmk:BS-cones}
Recall that when $\lambda = 0$, the Bryant--Salamon torsion-free $\G$-manifolds are in fact \emph{$\G$-cones}, so they are a particular case of Section~\ref{NKsec}. However, the `radial' coordinate $r$ in the descriptions~\eqref{eq:ph-BS-a} and~\eqref{eq:ph-BS-b} of these $\G$-structures is \emph{not} equal to the \emph{cone parameter}. In fact, this radial coordinate is a reparametrization of the cone parameter. This is consistent with Proposition~\ref{prop:NK-solitons}, which shows that for an isometric soliton on a $\G$-cone with cone parameter $\rho$, the soliton vector field $Y = b(\rho) \frac{\partial}{\partial \rho}$ must satisfy $b = c\rho$. In particular, if we let the cone parameter $\rho = \rho(r)$ in the Bryant--Salamon case, in terms of the radial coordinate $r$, then we have
$$ c \rho \frac{\partial}{\partial \rho} = \mu c r \frac{\partial}{\partial r} $$
where $\mu = 2$ in Case [a] and $\mu = \frac{3}{2}$ in Case [b]. It follows that, up to a constant factor, $\rho = r^{\frac{1}{\mu}}$.
\end{rmk}

Because of Remark~\ref{rmk:BS-cones}, we henceforth assume $\lambda > 0$, so $Y = b \frac{\partial}{\partial r} = 0$ and $c = 0$ by Proposition~\ref{prop:BS-LYg}.

\begin{prop} \label{prop:BS-solitons}
Let $\phir$ be the the torsion-free $\G$-structure $\varphi_{\lambda}$ given by~\eqref{eq:ph-BS-a} on $M = \Lambda^2_- (B^4)$ in Case [a], or by~\eqref{eq:ph-BS-b} on $M = \slashed{S} (S^3)$ in Case [b], with metric $g$ given by~\eqref{eq:general-bundle-metric} with data given by~\eqref{eq:hk-functions}. Assume that $\lambda > 0$. Let $\phii$ be the isometric $\G$-structure given by~\eqref{isog2form}, and let $(\phii,Y,c)$ satisfy the isometric soliton system~\eqref{solitoncond} where $X$ and $Y$ are assumed to be of the form~\eqref{BS-symm}. Then $c = 0$ and $Y = 0$ by Proposition~\ref{prop:BS-LYg}. If $a(r),f(r)$ are given by the trigonometric substitution \eqref{BS-trigsub} 
then $u(r)$ is a solution of the following ODE:
\begin{equation} \label{eq:BS-soliton-ODE}
\begin{aligned}
&\text{\rm Case [a]:} & \quad 
r^2 u'' + \dfrac{4\lambda+7r^2}{2(\lambda + r^2)} r u' - \dfrac{4\lambda^2+4\lambda r^2+3r^4}{2(\lambda + r^2)^2} \sin u &=0,\\
& \text{\rm Case [b]:} & \quad 
r^2 u'' + \dfrac{9\lambda +13r^2}{3(\lambda+r^2)} r u' - \dfrac{9\lambda^2+12\lambda r^2+8r^4}{3(\lambda+r^2)^2} \sin u &=0. 
\end{aligned} 
\end{equation}
Moreover the torsion of $\phii$ satisfies
\begin{equation} \label{eq:BS-soliton-torsion}
\begin{aligned}
& \text{\rm Case [a]:} & \quad |T|^2 & = (\lambda+r^2)^{\frac{1}{2}} \left( (u')^2 + \dfrac{4 \lambda^2 + 4 \lambda r^2 + 3r^4}{ r^2 (\lambda + r^2)^2}(1 - \cos u) \right), \\
& \text{\rm Case [b]:} & \quad |T|^2 & = (\lambda+r^2)^{\frac{1}{3}} \left( \dfrac{(u')^2}{4} + \dfrac{9 \lambda^2 + 12 \lambda r^2 + 8r^4}{6 r^2 (r^2 + \lambda)^2} (1 - \cos u) \right).
\end{aligned}
\end{equation}
\end{prop}
\begin{proof}
Since $Y = 0$, the second soliton condition in~\eqref{solitoncond} simplifies to $\tdiv T=0$. Because in both cases, $\Delta X$ is a multiple of $X$, by Corollary~\ref{cor:T-phidropout} we have
$$ (\tdiv T)^{\sharp} = 2(\Delta f) X - 2 f \Delta X. $$
Substituting in for $\Delta f$ and $\Delta X$ from Corollaries~\ref{cor:general-Laplacian} and~\ref{cor:general-Laplacian-1}, we have
\begin{align*}
(\tdiv T)^{\sharp} & = \dfrac2{k^2}\left[ a f'' + \left( \dfrac{m-1} r + n \dfrac{h'}{h} + (m-2) \dfrac{k'}{k} \right) a f' \right. \\ 
& \qquad {} - \left. \dfrac{2}{k^2} \left( f a'' + \left( \dfrac{m-1}{r} + n \dfrac{h'}{h} + m \dfrac{k'}{k} \right) f a' + \left( - \dfrac{m-1}{r^2} + n \left( \dfrac{h'}{h} \right)' + m \left( \dfrac{k'}{k} \right)' \right) a f \right) \right] \dib{r}
\end{align*}
which we rewrite as
\begin{align}
(\tdiv T)^{\sharp} & = \dfrac2{k^2}\left[ a f'' - f a'' - 2 \dfrac{k'}{k} a f' + \left( \dfrac{m-1}{r} + n \dfrac{h'} {h} + m \dfrac{k'}{k} \right) (a f' - f a') \right]\dib{r} \notag \\
& \qquad {} - \dfrac2{k^2}\left[ \left( - \dfrac{m-1}{r^2} + n \left( \dfrac{h'}{h} \right)' + m \left( \dfrac{k'}{k} \right)' \right) a f \right]\dib{r}. \label{BSfirstdivT}
\end{align}
Working with the trigonometric substitution~\eqref{BS-trigsub} yields $a f = \dfrac1{2k} \sin u$ and
\begin{align*}
a f' - f a' = & - \dfrac{1}{2k} u' + \dfrac{k'}{2k^2} \sin u, \\
a f'' - f a'' - 2 \dfrac{k'}{k} a f' = & \dfrac1{2k} \left( - u'' + 2 \dfrac{k'}{k} u' + \left( \dfrac{k''}{k} - 2 \dfrac{(k')^2}{k^2} \right) \sin u \right).
\end{align*}
Substituting these into~\eqref{BSfirstdivT} and simplifying gives
\begin{align*}
(\tdiv T)^{\sharp} & = -\dfrac{1}{k^3} \Bigg[ u'' - 2 \dfrac{k'}{k} u' - \left( \dfrac{k''}{k} - 2 \dfrac{(k')^2}{k^2} \right) \sin u + \left( \dfrac{m-1}{r} + n \dfrac{h'}{h} + m \dfrac{k'}{k} \right) \left( u' - \dfrac{k'}{k} \sin u \right) \\
& \qquad {} \qquad + \left( - \dfrac{m-1}{r^2} + n \left( \dfrac{h'}{h} \right)' + m \left( \dfrac{k'}{k} \right)' \right) \sin u \Bigg] \dib{r}.
\end{align*}
Substituting in the data from~\eqref{eq:hk-functions} for Case [a], further simplification yields
\begin{align*}
(\tdiv T)^\sharp & = - \dfrac{1}{k^3} \Bigg[ u'' + \dfrac{r}{\lambda+ r^2} u' + \dfrac{2 \lambda - r^2}{4 (\lambda+r^2)^2} \sin u + \dfrac{4 \lambda + 5 r^2}{2 r (\lambda+r^2)} \left( u' + \dfrac{r}{2 (\lambda + r^2)} \sin u \right) \\
& \qquad \qquad {} - \left( \dfrac{4 \lambda^2 + 7 \lambda r^2 + 5 r^4}{ 2r^2 (\lambda + r^2)^2} \right) \sin u \Bigg] \dib{r} \\
& = - \dfrac{1}{k^3} \left[ u'' + \dfrac{4 \lambda + 7 r^2}{2 r (\lambda + r^2)} u' - \left( \dfrac{4 \lambda^2 + 4 \lambda r^2 + 3 r^4}{2 r^2 (\lambda + r^2)^2} \right) \sin u \right] \dib{r}.
\end{align*}
Similarly, for Case [b] we compute that
\begin{align*}
(\tdiv T)^{\sharp} & = - \dfrac{1}{k^3} \Bigg[ u'' + \dfrac{2 r}{3 (\lambda + r^2)} u' + \dfrac{3 \lambda - 2 r^2}{9 (\lambda + r^2)^2} \sin u + \dfrac{9 \lambda + 11 r^2}{3 r (\lambda + r^2)} \left( u' + \dfrac{r}{3 (\lambda + r^2)} \sin u \right) \\
& \qquad \qquad {} - \left( \dfrac{9 \lambda^2 +16 \lambda r^2 + 11 r^4}{3 r^2 (\lambda + r^2)^2} \right) \sin u \Bigg] \dib{r} \\
& = - \dfrac{1}{k^3} \left[ u'' + \dfrac{9 \lambda + 13 r^2}{3 r (\lambda + r^2)} u' - \left( \dfrac{9 \lambda^2 + 12 \lambda r^2 + 8 r^4}{3 r^2 (\lambda + r^2)^2} \right) \sin u \right] \dib{r}.
\end{align*}
Setting each of these expressions equal to zero and clearing some denominators yields the ordinary differential equations in~\eqref{eq:BS-soliton-ODE}.

Recall from Proposition~\ref{prop:T-formula} that $\tfrac{1}{4} |T|^2 = |\nabla f|^2 + |\nabla X|^2$. From~\eqref{eq:nablas} and~\eqref{eq:normdbydr} we have
$$ \nabla f = \dfrac{f'}{k^2} \dib{r} \qquad \text{and hence} \qquad |\nabla f|^2 = \dfrac{(f')^2}{k^4} \left| \dib{r} \right|^2 = \dfrac{(f')^2}{k^2}. $$
The result of Corollary~\ref{cor:general-normsq-nablaX} can be rearranged as
$$ |\nabla X|^2 = \left( a' + a \dfrac{k'}{k} \right)^2 + (m-1) a^2 \left( \dfrac{1}{r^2} + \dfrac{k'}{rk} + 2 \dfrac{(k')^2}{k^2} - \dfrac{k''}{k} \right) + n a^2 \left( \dfrac{h'}{h} \dfrac{k'}{k} + \dfrac{(h')^2}{h^2} - \dfrac{h''}{h}\right). $$
Using the trigonometric substitution~\eqref{BS-trigsub} and the above two expressions, we obtain
\begin{align*}
\tfrac{1}{4} |T|^2 & = |\nabla f|^2 + |\nabla X|^2 \\
& = \dfrac{(u')^2}{4 k^2} + (m-1) \dfrac{(1 - \cos u)}{2k^2} \left( \dfrac{1}{r^2} + \dfrac{k'}{r k} + 2 \dfrac{(k')^2}{k^2} - \dfrac{k''}{k} \right) + n \dfrac{(1 - \cos u)}{2k^2} \left( \dfrac{h'}{h} \dfrac{k'}{k} + \dfrac{(h')^2}{h^2} - \dfrac{h''}{h} \right).
\end{align*}
Substituting in the data from \eqref{eq:hk-functions} in each case yields the expressions~\eqref{eq:BS-soliton-torsion} for $|T|^2$.
\end{proof}

In Section~\ref{ODEsec} we analyze the solutions of ODEs in a class including \eqref{eq:BS-soliton-ODE}, allowing us to state the following conclusions:

\begin{thm} \label{thm:BS-solitons-final}
There exist smooth solutions to \eqref{eq:BS-soliton-ODE} which are odd functions of $r$ defined for all $r \in \R$. Thus, the steady isometric soliton $\varphi$, where $\varphi$ is as in Proposition~\ref{prop:BS-solitons}, is smooth on $M$. Moreover, the torsion of $\varphi$ satisfies $\lim_{r \to \infty} |T|^2 = 0$.
\end{thm}
\begin{proof}
Global existence of odd ODE solutions $u(r)$ is established in Section~\ref{sec:globalexist}. Then, since $X$ given by~\eqref{BS-symm} and $f(r) = \cos( \frac{1}{2} u(r))$ are smooth on $M$, we deduce that $\varphi_{\lambda}$ is smooth. The asymptotic behavior of $|T|$ is established in the last part of Section~\ref{sec:torlimit}.
\end{proof}

\section{Existence and behavior of ODE solutions} \label{ODEsec}

In this section, we establish local existence of solutions (in a neighbourhood of the regular singular point $r=0$) for a class of ordinary differential equations of the form
\begin{equation} \label{ODEclass}
r^2 u'' + \dfrac{d-c r^2}{1+br^2} r u' - \dfrac{d+er^2+f r^4}{(1+b r^2)^2} \sin u = 0,
\end{equation}
where $b,c,d,e,f$ are constants with $d>0$ and $b,e,f\ge 0$. This class includes the
ODE~\eqref{firstradialODE} for $b=e=f=0$ and $d=6$, as well as the ODEs~\eqref{eq:BS-soliton-ODE} describing radially symmetric soliton metrics on Bryant--Salamon backgrounds. (We note for use below that $f > 0$ in the Bryant--Salamon cases.) We also establish global existence of solutions, and investigate their asymptotic behavior as $r \to +\infty$.

To see that $r=0$ qualifies as a regular singular point for this nonlinear equation, one can rewrite the ODE as a first-order system. Letting $z=r u'$, hence $z' = u' + ru''$, we obtain
\begin{equation} \label{singularsys}
\begin{aligned}
r u' &= z, \\
r z' & = \left( 1 - \dfrac{d-c r^2}{1+br^2}\right) z +\dfrac{d+er^2+f r^4}{(1+b r^2)^2} \sin u.
\end{aligned}
\end{equation}
Because the right-hand sides are analytic near $r=0$, this meets the definition of regular singular point for a nonlinear system. (See Wasow~\cite[page 200]{wasow}, where these are also called {\em singularities of the first kind}, the term also used by Hsieh--Sibuya~\cite{hsisub}.)

\subsection{Local existence}
\begin{thm} There exists a 1-parameter family of analytic odd solutions to~\eqref{ODEclass}, each analytic on an open set containing $r=0$, and parametrized by the value of $u'(0)$.
\end{thm}
\begin{proof}
We will prove that there exist formal power series solutions $u=\sum_{k=1}^{\infty} a_k r^k$ to~\eqref{ODEclass} with $a_1\ne 0$ and which contain only odd powers of $r$. Then it will follow by Hsieh--Sibuya~\cite[Theorem V-2-7]{hsisub} that these series converge to a solution on some open interval around $r=0$. The presence of the nonlinear factor $\sin u$ is handled by the following:

\begin{lemma} \label{sinelemma}
If $u=\sum_{k=1}^{\infty} a_k r^k$ is an analytic function, then for $k>2$ there there are polynomial functions $T_k(a_1, \ldots, a_{k-2})$ such that
$$ \sin u = \sum_{k=1}^{\infty} s_k r^k, $$
where $s_k = a_k + T_k$ and $T_1 = T_2 = 0$. Moreover, when $a_2 = a_4 = \cdots = a_{2j} = 0$, then $T_{2j + 2} = 0$.
\end{lemma}
This lemma will be proved below, after the proof of the theorem is complete.

Substituting $u=\sum_{k=1}^{\infty} a_k r^k$ into the ODE~\eqref{ODEclass} and clearing denominators yields
\begin{align*}
0 & = (1+b r^2)^2 \sum_{k=1}^{\infty} k(k-1) a_k r^k + (1+br^2)(d-cr^2) \sum_{k=1}^{\infty} k a_k r^k - (d+er^2+f r^4) \sum_{k=1}^{\infty} s_k r^k \\
& = \sum_{k=1}^{\infty} \Big[ \left[ (k+d)(k-1) a_k - d T_k\right] r^k + \left[ (2bk(k-1) + (bd-c)k-e) a_k -e T_k\right] r^{k+2} \\
& \qquad \qquad {} + \left[ (b^2 k(k-1) -bck-f) a_k - f T_k\right] r^{k+4} \Big].
\end{align*}

In this sum, the coefficient of $r^1$ vanishes (since $T_1=0$), while setting the coefficient of $r^2$
equal to zero gives $0 = (2+d) a_2$ (since $T_2=0$), implying that any series solution must have $a_2=0$. The coefficients of $r^3$ and $r^4$ yield the equations
\begin{equation} \label{a4qn}
\begin{aligned}
0 &= 2(d+3)a_3 - d T_3 + (bd-c-e) a_1, \\
0 &= 3(d+4)a_4 - d T_4 + (4b+2(bd-c)-e)a_2,
\end{aligned}
\end{equation}
and in general for $k>4$, we have the recurrence relation
\begin{align} \nonumber
& (k-1)(k+d) a_k + (2b(k-2)(k-3) +(bd-c)(k-2)-e) a_{k-2} + (b^2 (k-4)(k-5) -bc(k-4)-f) a_{k-4} \\ \label{arecurrence}
& \qquad = d T_k + e T_{k-2} +f T_{k-4}.
\end{align}
Recalling that $T_k = T_k(a_1, \ldots, a_{k-2})$, we see that $a_k$ for $k\ge 3$ is uniquely determined by $a_1, \ldots, a_{k-2}$.

By Lemma~\ref{sinelemma}, since $a_2=0$ then $T_4=0$ and hence~\eqref{a4qn} gives $a_4=0$. More generally, if $a_2 = a_4 = \ldots = a_{2j}=0$ for some $j>1$, then Lemma~\ref{sinelemma} implies that $T_2 = \ldots = T_{2j+2}=0$, so the recurrence relation~\eqref{arecurrence} for $k=2j+2$ implies that $a_{2j+2}=0$.
\end{proof}

\begin{proof}[Proof of Lemma~\ref{sinelemma}]
Substituting the series for $u$ into the Maclaurin series for the sine function gives 
$$ \sin u = \sum_{n=0}^{\infty} \dfrac{(-1)^n}{(2n+1)!} \left(\sum_{m=1}^{\infty} a_m r^m\right)^{2n+1}. $$
Since each factor $\sum_{m=1}^{\infty} a_m r^m$ has degree at least one in $r$, we observe that for a fixed $k$ the coefficient of $r^k$ on the right-hand side only involves the $a_m$ for which $m+2n \le k$. Noting the special form for the $n=0$ term, we can write
$$ \sum_{k=1}^{\infty} s_k r^k = \sum_{k=1}^{\infty} a_k r^k + \sum_{n=1}^{\infty} \dfrac{(-1)^n}{(2n+1)!} \left(\sum_{m=1}^{\infty} a_m r^m\right)^{2n+1}. $$
Due to our observation, and $m+2n\le k$ implying $2n\le k-1$, we have
\begin{equation} \label{sinecoeffs}
s_k = a_k + \, \, \text{the $r^k$ coefficient in} \, \, \sum_{n=1}^{\lfloor (k-1)/2\rfloor} \dfrac{(-1)^n}{(2n+1)!} \left(\sum_{m=1 }^{k-2n} a_m r^m\right)^{2n+1}.
\end{equation}
In the case $k=1$ and $k=2$, the second term is zero, so $s_1=a_1$ and $s_2=a_2$. In general, since in the second term we take the $r^k$ coefficient of a finite degree polynomial in $r$, whose coefficients are polynomial in the $a_m$ for $m \le k-2n \le k-2$, then 
$$ s_k = a_k + T_k(a_1, \ldots, a_{k-2}), $$
as claimed.

Lastly, if $k=2j+2$ for $j\ge 1$ and $a_2 = \ldots = a_{2j}=0$, then only odd powers of $r$ are present in the second term of~\eqref{sinecoeffs}, and so $T_{2j+2}=0$.
\end{proof}

\subsection{Global existence}\label{sec:globalexist}

Away from the regular singular point, we can also establish existence of solutions to~\eqref{ODEclass} defined for all positive values of $r$. If we make the substitution $r = \re^x$, the ODE becomes
\begin{equation} \label{ODEclassx}
\dfrac{d^2 u}{dx^2} + (P(\re^x)-1) \dfrac{du}{dx} -Q(\re^x) \sin u = 0, \qquad \text{where }
P(r) = \dfrac{d-c r^2}{1+br^2}, \ Q(r) = \dfrac{d+e r^2+f r^4}{(1+br^2)^2}.
\end{equation}
In order to apply standard existence theorems, we rewrite this as a system in vector form $d \vu / dx = \vF (\vu,x)$, where $\vu = (u, z)^t$ and
\begin{equation} \label{system-eq}
\vF(\vu,x) = \begin{pmatrix} z \\ Q(\re^x) \sin u + (1-P(\re^x)) z \end{pmatrix}.
\end{equation}
The following result is easily checked:

\begin{lemma}
Let $I$ be any bounded open interval in $\R$, and let $D = \R^2$. Then $\vF$ satisfies a Lipschitz estimate on $D \times I$. That is, there is a constant $L$ such that
for any $\vu_1, \vu_2\in D$ and $x \in I$, we have
$$ |\vF(\vu_1,x) - \vF(\vu_2,x)| \le L |\vu_1-\vu_2|. $$
\end{lemma}

It now follows from~\cite[Theorem 1.7]{grimshaw} that, for initial values posed at $x_0 \in I$, there exists a solution $\vu(x)$ defined for $x \in [x_0, x_1)$, and if $x_1 < \sup I$, then $|\vu(x)| \to \infty$ as $x \to x_1$. However, since $|\sin u|\le 1$, then $z=du/dx$ satisfies
$$ (1-P) z - Q \le \dfrac{dz}{dx} \le (1-P)z + Q. $$
(Recall that $Q=Q(\re^x)$ is strictly positive.) Thus, for all $x \in [x_0, x_1)$, we have
$$ z_1(x) \le z(x) \le z_2(x), $$
where $z_1, z_2$ are solutions of the linear ODEs $dz_1/dx = (1-P) z_1 -Q$ and $dz_2/dx = (1-P) z_2 +Q$ with the same initial values. Since $u(x) = \int_{x_0}^{x} z(t) \,dt$,
it follows that $|\vu(x)|$ is bounded as $x\to x_1$. Applying the same reasoning for $x \le x_0$, we see that solution $\vu(x)$ exists for all $x\in I$. Since $I$ was arbitrary, it follows 
that solutions to a given initial value problem for~\eqref{ODEclassx} exists for all real $x$.

\subsection{Asymptotic behavior of torsion}\label{sec:torlimit}

\textbf{Flat and Nearly-K\"ahler Cases.} First, we consider the asymptotic behavior of solutions to~\eqref{ODEclassx} in the special case where $b=e=f=0$. (This covers both the case of solitons constructed relative to a flat background on $\R^7$ as well as those relative to fibrations over a nearly-K\"ahler metric. In both cases, $d=6$.) The ODE specializes to 
\begin{equation} \label{flat:ODEnodenoms}
r^2 u'' + (6-cr^2) r u' - 6 \sin u =0.
\end{equation}
After the change of variable $r=\re^x$ and $z=du/dx = r u'$, we have
$$\dfrac{dz}{dx} = \left( ce^{2x} - (d -1)\right) z + d \sin u.$$
Let $L = \tfrac{1}{2} z^2 + d \cos u$. Then we compute
\begin{equation} \label{Lderivsimple}
\dfrac{dL}{dx} = \left( ce^{2x} - (d -1)\right) z^2.
\end{equation}
Thus, if $c \le 0$ then we immediately have $dL/dx \le 0$, so $|z|$ is bounded as $r \to +\infty$. Thus, $\lim_{r \to + \infty} u' = 0$ and the expression~\eqref{flat:squarenormT} shows that $\lim_{r\to \infty} |T|=0$.

For the remainder of our discussion of these cases, we assume that $c>0$, and $u(r)$ is a nonzero solution of~\eqref{flat:ODEnodenoms} defined for all $r \in \R$. 
\begin{lemma}[\cite{BWharmonic}] \label{lem:BWblowup}
Under these assumptions $\lim_{r\to+\infty} r^3 |u'| = \infty$.
\end{lemma}
\begin{proof}
Our ODE~\eqref{flat:ODEnodenoms} is a special case of~\cite[Equation (7)]{BWharmonic} when their parameter $d$ is set equal to 7, and one makes the change of variables $y = \sqrt{2c}\, r$ and $f(y) = \tfrac{1}{2} u(r)$. The conclusion then follows from the argument in the proof of~\cite[Theorem 1]{BWharmonic}.
\end{proof}

We now prove that the torsion is in fact unbounded as $r \to +\infty$, by showing that boundedness of the torsion would imply a contradiction to Lemma~\ref{lem:BWblowup}. First, we rewrite~\eqref{flat:ODEnodenoms} in another way as a system. By inserting integrating factors into the linear terms, we find that~\eqref{flat:ODEnodenoms} is equivalent to
$$ \dfrac{d}{dr}\left( r^6 e^{-\tfrac{1}{2} cr^2} u'\right) = 6 r^4 e^{-\tfrac{1}{2} cr^2} \sin u. $$
Now let
$$ w = \tfrac16 r^6 e^{-\tfrac{1}{2} cr^2} u', $$
so that $u' = 6 r^{-6} e^{\tfrac{1}{2} cr^2} w$ and $w' = r^4 e^{-\tfrac{1}{2} cr^2} \sin u$.

\def\rlim{\lim_{r \to +\infty}}
\def\rliminf{\underset{r \to +\infty} {\lim\inf}}

\begin{prop} \label{prop:Hopital}
Suppose that $\rliminf\, |T| < \infty$. Then for any sequence $r_n \to +\infty$, $\lim_{n\to \infty} r_n^3 u'(r_n)$ converges along a subsequence.
\end{prop}
\begin{proof}
Our proof essentially follows the argument in~\cite[Remark 1]{BWharmonic}. From~\eqref{flat:squarenormT}, our hypothesis implies that $\rliminf |u'| < \infty$. Since $u' = 6 r^{-6} e^{\tfrac{1}{2} cr^2} w$, we necessarily have $\rlim w(r) = 0$. But 
$$ r^3 u'(r) = \dfrac{w(r)}{\ell(r)}, \qquad \text{where} \quad \ell(r) = \tfrac{1}{6} r^3 e^{-\tfrac{1}{2} cr^2}. $$
Since $\rlim w(r) = \rlim \ell(r) = 0$, then $\rlim r^3 u'(r)$ could be assessed using l'H\^{o}pital's rule. In fact, 
$$ \dfrac{w'(r)}{\ell'(r)} = \dfrac{r^4 e^{-\tfrac{1}{2} cr^2} \sin u}{\tfrac{1}{6} (3r^2 - cr^4)e^{-\tfrac{1}{2} cr^2}} =\dfrac{6 \sin u}{3r^{-2} - c}$$
is bounded as $r \to +\infty$. So, for any sequence $r_n\to \infty$, there exists a subsequence for which the right hand side converges.
Hence by the proof of l'H\^{o}pital's rule, $r_n^3 u'(r_n)$ converges along a subsequence.
\end{proof}

From the contradiction between Lemma~\ref{lem:BWblowup} and the conclusion of Proposition~\ref{prop:Hopital}, we deduce that $ \rliminf \, |T| = \infty$. This completes the proof of Proposition~\ref{prop:asymptotic-torsion}.

\textbf{Bryant--Salamon Cases.} Next, we consider the case where $b>0$ and $c<0$. For convenience of notation, let $A = 1 + b r^2$ and $B = d + er^2 + f r^4$, so by~\eqref{ODEclassx} we have $P = \frac{d - cr^2}{A}$ and $Q = \frac{B}{A^2}$. If we define
\begin{equation} \label{eq:L-defn}
L = \tfrac{1}{2} A^2 z^2 + B \cos u,
\end{equation}
then using~\eqref{system-eq} and $r = \re^x$, we compute that
\begin{align} \nonumber
\dfrac{dL}{dx} & = A \dfrac{dA}{dx} z^2 + A^2 z \left( \left(1 - \dfrac{d - c e^{2x}}{A} \right) z + \dfrac{B}{A^2} \sin u \right) + \dfrac{dB}{dx} \cos u - B z \sin u \\ \nonumber 
& = A^2 z^2 \left(\dfrac{1}{A} \dfrac{dA}{dx} + 1 - \dfrac{d - c e^{2x}}{A} \right) + \dfrac{1}{B} \dfrac{dB}{dx} (L - \tfrac{1}{2} A^2 z^2) \\ \label{bottomLline}
& = A^2 z^2 \left[\dfrac{1}{A} \dfrac{dA}{dx} + 1 - \dfrac{d - c e^{2x}}{A} - \dfrac{1}{2B} \dfrac{dB}{dx} \right] + \dfrac{1}{B} \dfrac{dB}{dx} L.
\end{align}

\begin{lemma} \label{myBSlemma}
Given $\varepsilon \in (0,1)$ and $\delta \in (0,1)$, there exists an $m>0$ and $M \in \R$ and such that if $x\ge M$ and $z^2 \ge m$ then
\begin{equation} \label{eq:myBSlemma}
\dfrac{dL}{dx} \le (1+\varepsilon)\left[ 4+2(1-\delta) (1+c/b) \right] L.
\end{equation}
\end{lemma}
\begin{proof}
Note that $\displaystyle\lim_{x \to +\infty} B^{-1} \frac{dB}{dx} = 4$. Thus, given $\varepsilon \in (0,1)$ there is an $M\in \R$ such that 
$$ 4(1-\varepsilon) \le \dfrac{1}{B} \dfrac{dB}{dx} \le 4(1+\varepsilon) $$
for all $x \ge M$. Since the coefficient in square brackets in~\eqref{bottomLline} has limit $1+c/b$ as $x \to \infty$, which is negative, we choose $M$ large enough that we also have
$$ (1+c/b) (1 + \varepsilon) \le \dfrac{1}{A} \dfrac{dA}{dx} + 1 - \dfrac{d - c e^{2x}}{A} - \dfrac{1}{2B} \dfrac{dB}{dx} \le (1 + c/b) (1-\varepsilon) $$
for all $x\ge M$. Moreover, if we want 
$$ A^2 z^2 \ge 2(1-\delta) L = (1-\delta) A^2 z^2 + 2(1-\delta) B \cos u $$
then this would be implied by
$$ z^2 \ge 2 \dfrac{1-\delta}{\delta} \dfrac{B}{A^2}. $$
Recall that $f > 0$ in these cases. Since $\displaystyle\lim_{x \to \infty} (B/A^2) = f/b^2$, we can further increase $M$ to ensure that $B/A^2 \le (1+\varepsilon) f/b^2$ when $x \ge M$. Then for
$$ m = 2\dfrac{1-\delta}{\delta}(1+ \varepsilon) \dfrac{f}{b^2}, $$
the inequality $z^2 \ge m$ implies
$$ A^2 z^2 \ge 2 (1-\delta)L $$
and consequently putting all of these inequalities together and using~\eqref{bottomLline} we conclude that
$$ \dfrac{dL}{dx} \le 2(1-\delta) (1+\varepsilon)(1+c/b) L + 4(1+\varepsilon)L $$
as claimed.
\end{proof}

In the Bryant--Salamon cases we have $c/b = -13/3$ or $c/b=-7/2$, so that $2(1+c/b) < -4$. Thus, in either case we can (and will) choose $\delta$ so that the coefficient on the right hand side of~\eqref{eq:myBSlemma} is negative. Thus, for $x$ sufficiently large, whenever $z^2 \ge m$ then $L$ is decreasing. However, as the following lemma shows, $z^2$ is bounded above for $x$ sufficiently large.

\begin{lemma} Let $m$, $M$ be as in Lemma~\ref{myBSlemma}. There are constants $m' \ge m$ and $M'\ge M$ such that $z^2 \le m'$ whenever $x \ge M'$.
\end{lemma}
\begin{proof}
Choose $M' \ge M$ large enough so that there are positive constants $c_1 < c_2$ and $c_3 < c_4$ satisfying
\begin{equation} \label{seebounds}
c_1 \le \tfrac12 A^2 e^{-4x} \le c_2, \qquad c_3 \le B e^{-4x} \le c_4
\end{equation}
for $x \ge M'$. Let $E = \{ x \ge M':\text{ there exists $\beta > 0$ such that $z^2 \geq m$ on $[x, x+\beta)$} \}$. If $E$ is empty, then $z^2 \leq m$ on $[M', \infty)$, so our conclusion holds with $m' = m$. If $E$ is not empty, let $x_0 = \min E$, and let $\beta > 0$ be such that $z^2 \ge m$ on $[x_0, x_0 + \beta)$. Let $m_0 \geq m$ denote the value of $z^2$ at $x = x_0$. Lemma~\ref{myBSlemma} implies that $L$ is decreasing on $[x_0, x_0 + \beta)$. The definition of $L$ gives $ \tfrac{1}{2} A^2 e^{-4x} z^2 = L e^{-4x} - B e^{-4x} \cos u$ and hence
$$ c_1 z^2 \le L e^{-4x} + c_4. $$
But since $L$ is decreasing, then on this interval $L e^{-4x}$ is bounded above by its value at $x = x_0$, which is
$$ \tfrac{1}{2} A^2 e^{-4x_0} m_0 + B e^{-4x_0} \cos u \le c_2 m_0 + c_4. $$
Hence on this interval, we have $z^2 \le (c_2 m_0 + 2c_4)/c_1 = m'$. If we can take $\beta$ arbitrarily large, then we are done. If not, there exists a maximal $\beta > 0$ such that $z^2 = m$ at $x=x_0+\beta$ but $z^2 < m$ on some open interval beginning at $x_0+\beta$. 
If $z^2$ ever rises above $m$ again, then we apply the same argument, with $L$ decreasing on some interval
beginning at a point where $z^2=m$, and obtain the bound $z^2 \le (c_2 m +2c_4)/c_1 \le m'$.
\end{proof}

In the Bryant-Salamon cases we have computed the square norm of the torsion in~\eqref{eq:BS-soliton-torsion}. In terms of $z$, these are:
\begin{align*}
&\text{\rm Case [a]:}& \quad 
|T|^2 & = \dfrac{ (\lambda+r^2)^{\frac{1}{2}}}{r^2} z^2 + \dfrac{4\lambda^2+4\lambda r^2+3r^4}{r^2 (\lambda+r^2)^{3/2}}(1-\cos u) \\
&\text{\rm Case [b]:}& \quad 
|T|^2 &= \dfrac{(\lambda+r^2)^{\frac{1}{3}}}{4 r^2} z^2 + \dfrac{9\lambda^2+12\lambda r^2+8r^4}{6r^2 (\lambda+r^2)^{5/3}} (1-\cos u).
\end{align*}
Because the $r$-dependent coefficients in these formulas all have limit zero as $r \to +\infty$, using the fact that $z^2$ is bounded as $r\to \infty$, we conclude that $\underset{r \to +\infty}{\lim} |T|^2 = 0$, completing the proof of Theorem~\ref{thm:BS-solitons-final}.


\addcontentsline{toc}{section}{References}

\begin{thebibliography}{99}

\bibitem{Bag} L. Bagaglini, ``The energy functional of $\G$-structures compatible with a background metric'', {\em J. Geom. Anal.} {\bf 31} (2021), 346–-365.

\bibitem{BWharmonic} P.\ Bizo\'n\ and\ A.\ Wasserman, ``Nonexistence of shrinkers for the harmonic map flow in higher dimensions'' {\em Int. Math. Res. Not. IMRN} (2015), 7757–7762. MR3403999

\bibitem{Bryant} R.L. Bryant, Some remarks on $\G$-structures, in {\it Proceedings of G\"{o}kova Geometry-Topology Conference 2005}, 75--109, G\"{o}kova Geometry/Topology Conference (GGT), G\"{o}kova. MR2282011

\bibitem{BS} R.L. Bryant\ and\ S.M. Salamon, ``On the construction of some complete metrics with exceptional holonomy'', {\em Duke Math. J.} {\bf 58} (1989), 829--850.

\bibitem{DGK} S. Dwivedi, P. Gianniotis\ and\ S. Karigiannis, ``A gradient flow of isometric $\G$-structures'', {\em J. Geom. Anal.} {\bf 31} (2021), 1855--1933. MR4215279

\bibitem{Flows2} S. Dwivedi, P. Gianniotis\ and\ S. Karigiannis, ``Flows of $\G$-structures, II: Curvature, torsion, symbols, and functionals'', {\em preprint}, arXiv:2311.05516

\bibitem{DLS} S. Dwivedi, E. Loubeau, and H. S\'a Earp, ``Harmonic flow of $\mathrm{Spin}(7)$-structures'', {\em Ann. Sc.
Norm. Super. Pisa Cl. Sci. (5)} {\bf 25}, 151--215. MR4732637

\bibitem{FLMS} D. Fadel, E. Loubeau, A.J. Moreno, and H. S\'a Earp, ``Flows of geometric structures'', {\em preprint}, arXiv:2211.05197

\bibitem{FS} U. Fowdar and H. S\'a Earp, ``Harmonic flow of quaternion-K\"ahler structures'', {\em preprint}, arXiv:2301.12494

\bibitem{Grigorian-Oct} S. Grigorian, ``$\G$-structures and octonion bundles'', {\em Adv. Math.} {\bf 308} (2017), 142--207. MR3600058

\bibitem{G-IF} S. Grigorian, ``Estimates and monotonicity for a heat flow of isometric $\G$-structures'', {\em Calc. Var. Partial Differential Equations} {\bf 58} (2019), Paper No.\ 175, 37 pp. MR4018307

\bibitem{Grigorian-IF-survey} S. Grigorian, ``Isometric flows of $\G$-structures'', {\it Current trends in analysis, its applications and computation}, 545--553. Trends Math. Res. Perspect. Birkhäuser/Springer, Cham, 2022. MR4559550

\bibitem{grimshaw} R. Grimshaw, {\it Nonlinear ordinary differential equations}, Applied Mathematics and Engineering Science Texts. CRC Press, Boca Raton, FL, 1993.

\bibitem{hsisub} P.-F. Hsieh\ and\ Y. Sibuya, {\it Basic Theory of Ordinary Differential Equations}, Universitext. Springer-Verlag, New York, 1999.

\bibitem{K-thesis} S. Karigiannis, ``Deformations of $\G$ and $\Spin{7}$ structures'', {\em Canad. J. Math.} {\bf 57} (2005), 1012--1055. MR2164593

\bibitem{K-flows} S. Karigiannis, ``Flows of $\G$-structures. I'', {\em Q. J. Math.} {\bf 60} (2009), 487--522. MR2559631

\bibitem{KL} S. Karigiannis\ and \ J.D. Lotay, ``Bryant--Salamon $\G$ manifolds and coassociative fibrations'', {\em J. Geom. Phys.} {\bf 162} (2021), Paper No. 104074, 60 pp.

\bibitem{KMT} S. Karigiannis,\ B. McKay,\ and\ M.-P. Tsui, ``Soliton solutions for the Laplacian co-flow of some $\G$-structures with symmetry'', {\em Differential Geom. Appl.} {\bf 30} (2012), 318--333.

\bibitem{wasow} W. Wasow, {\it Asymptotic Expansions for Ordinary Differential Equations}, reprint of the 1976 edition. Dover Publications, New York, 1987.

\end{thebibliography}
\end{document}